\documentclass{amsart}
\usepackage{amssymb,amsmath}
\usepackage{hyperref}
\usepackage{appendix}
\usepackage{paralist}
\usepackage{graphics}
\usepackage{epsfig}
\usepackage{graphicx}
\usepackage{epstopdf}
\usepackage[]{algorithm2e}
\usepackage{graphicx,amsxtra,amssymb, tikz}
\usepackage{amsmath}
\usepackage{epsfig}
\usepackage{hyperref}
\usepackage{appendix}
\usepackage{setspace}
\usepackage{MnSymbol,wasysym}
\newtheorem{thm}{Theorem}[section]
\newtheorem{remm}{Remark}[section]
\newtheorem{lem}{Lemma}[section]

\newtheorem{remark}[remm]{Remark}

\newtheorem{proposition}[thm]{Proposition}

\numberwithin{equation}{section}
\newcommand{\rset}{{\mathbb R}}
\newcommand{\nset}{{\mathbb N}}

\newcommand{\ken}{\ \ }

\newcommand{\btheta}{\boldsymbol\theta}
\newcommand{\brho}{\boldsymbol\rho}
\newcommand{\bhta}{\boldsymbol\eta}
\newcommand{\ee}{{\bf e}}
\newcommand{\ssy}{\scriptscriptstyle}

\newcommand{\uu}{u}
\newcommand{\UUU}{U}

\newcommand{\gff}{{\mathfrak n}}
\newcommand{\mfg}{{\mathfrak g}}

\newcommand{\half}{\frac{1}{2}}
\newcommand{\quarter}{\frac{1}{4}}
\newcommand{\dspace}{{\mathfrak X}^{\circ}_{h}}
\newcommand{\gspace}{{\mathfrak X}_h}
\newcommand{\sspace}{{\mathfrak S}_h}
\newcommand{\ddelta}{\delta_{\star}}

\newcommand{\PP}{{\sf I}^{\circ}_h}
\begin{document}
\title[]
{Error Estimation of the Besse Relaxation Scheme\\
 for a Semilinear Heat Equation}
\author[]{Georgios E. Zouraris$^{\ddag}$}
\thanks
{$^{\ddag}$ Department of Mathematics and Applied Mathematics,
University of Crete, GR-700 13 Panepistimioupolis, Heraklion, Crete, Greece.}
%
%
%
%
\begin{abstract}
%
%
The solution to the initial and Dirichlet
boundary value problem for a semilinear, one dimensional heat
equation is approximated by a numerical method that combines the
Besse relaxation scheme in time
(C. R. Acad. Sci. Paris S{\'e}r. I, vol. 326 (1998))
with a central finite difference method in space.
A new, composite stability argument is developed, leading to
an optimal, second-order error estimate in the discrete
$L_t^{\infty}(H_x^1)-$norm.
%
%
It is the first time in the literature where an error estimate for fully
discrete approximations based on the Besse relaxation scheme is
provided.
%
%
\end{abstract}
\maketitle
%
%
%
%
%
%
%
%
\section{Introduction}
\subsection{Formulation of the problem}
Let $T>0$, $x_a$, $x_b\in{\mathbb R}$ with $x_b>x_a$,
${\mathcal I}:=[x_a,x_b]$
and $\uu:[0,T]\times{\mathcal I}\rightarrow{\mathbb R}$
be the solution of the following initial and boundary value problem:
\begin{gather}
\uu_t=\uu_{xx}+g(\uu)\,\uu+f
\quad\text{\rm on}\ \ [0,T]\times{\mathcal I},
\label{PSL_a}\\
\uu(t,x_a)=u(t,x_b)=0\quad\forall\,t\in[0,T],
\label{PSL_b}\\
\uu(0,x)=\uu_0(x)\quad\forall\,x\in{\mathcal I},
\label{PSL_c}
\end{gather}
where $g:{\mathbb R}\rightarrow{\mathbb R}$,
$f:[0,T]\times{\mathcal I}\rightarrow{\mathbb R}$
and $\uu_0:{\mathcal I}\rightarrow{\mathbb R}$ with
\begin{equation}\label{PSL_d}
u_0(x_a)=u_0(x_b)=0.
\end{equation}
%
%
Furthermore, we assume that the data $f$, $u_0$ and $g$ are smooth enough
and compatible, in order to guarantee the existence and uniqueness of a solution
$u$ to the problem above that is sufficiently smooth for our purposes.
\par
Two decades ago, for the discretization in time of the nonlinear
Schr{\"o}dinger equation, C. Besse \cite{Besse1} introduced a new
linear-implicit time-stepping method (called {\it Relaxation Scheme})
as an attempt to avoid the numerical solution of the nonlinear systems
of algebraic equations that the application of the implicit Crank-Nicolson method yields.
The proposed time discretization technique, combined with a finite element or
a finite difference space discretization, is computationally efficient
(see, e.g., \cite{ABB2013}, \cite{Mitsotakis}, \cite{Henning})
and performs as a second order method (see, e.g., \cite{Besse2}, \cite{Mitsotakis}).
Later, C. Besse \cite{Besse2} analyzing the Relaxation Scheme as
a semidiscrete in time method to approximate the solution of the
Cauchy problem (i.e. without the presence of boundary conditions) shows,
using  that it is local well-posedness and convergent without concluding
a convergent rate with respect to the time-step.
Until today, in spite of the results in \cite{Besse2}, there is
no scientific work in the literature providing an error estimate
for the Relaxation Scheme. Since the Relaxation Scheme can not be
classified as a Runge-Kutta or a linear multistep method, a natural
question arises: ``is the Relaxation Scheme a special method or a
representative member of a new family of linear implicit time-discretization
methods?" One way moving toward to find an answer is first to understand
its convergence and then to construct methods with similar characteristics.
%
%
\par
The aim of the work at hands is to contribute to the understanding of the
convergence nature of the Besse relaxation scheme, by investigating its
use, along with a finite difference space discretization, to obtain
approximations of the solution to the parabolic problem \eqref{PSL_a}-\eqref{PSL_d}.
%
%
By building up a proper stability argument and using energy techniques,
we are able to prove an optimal, second order error estimate in a discrete
$L_t^{\infty}(H_x^1)-$norm. The result is new and opens the discussion
on the applicability and the extension of the Relaxation Scheme to other
non-linear evolution equations.
%
%
%
\subsection{Formulation of the numerical method}
\subsubsection{Notation}
Let $\nset$ be the set of all positive integers and ${\sf L}:=x_b-x_a$. For given
$N\in\nset$, we define a uniform partition of the time interval $[0,T]$ with
time-step $\tau:=\tfrac{T}{N}$, nodes $t_n:=n\,\tau$ for $n=0,\dots,N$,
and intermediate nodes $t^{n+\half}=t_n+\tfrac{\tau}{2}$ for
$n=0,\dots,N-1$. Also, for given $J\in\nset$, we consider a uniform partition
of ${\mathcal I}$ with mesh-width $h:=\tfrac{{\sf L}}{J+1}$ and nodes
$x_j:=x_a+j\,h$ for $j=0,\dots,J+1$. Then, we introduce the discrete spaces
\begin{equation*}
\gspace:=\left\{\,(v_j)_{j=0}^{\ssy J+1}:\,\, v_j\in\rset,\,\, j=0,\dots,J+1\,\right\}
\ken\text{\rm and}\ken
{\dspace}:=\left\{\,(v_j)_{j=0}^{\ssy J+1}\in\gspace:\,\, v_0=v_{\ssy J+1}=0\,\right\},
\end{equation*}
a discrete product operator
$\cdot\otimes\cdot:\gspace\times\gspace\rightarrow\gspace$ by
\begin{equation*}
(v{\otimes}w)_j=v_j\,w_j,\quad j=0,\dots,J+1,
\quad\forall\,v,w\in\gspace,
\end{equation*}
and a discrete Laplacian operator $\Delta_h:\dspace\rightarrow\dspace$ by
\begin{equation*}
\Delta_hv_j:=\tfrac{v_{j-1}-2v_j+v_{j+1}}{h^2},
\quad j=1,\dots,J,\quad\forall\,v\in\dspace.
\end{equation*}
In addition, we introduce o\-pe\-rators
${\sf I}_h:{\sf C}({\mathcal I};{\mathbb R})\rightarrow\gspace$
and
$\PP:{\sf C}({\mathcal I};{\mathbb R})\rightarrow\dspace$, which, for
given $z\in{\sf C}({\mathcal I};\rset)$, are defined by
$({\sf I}_h z)_j:=z(x_j)$ for $j=0,\dots,J+1$ and $z\in{\sf C}({\mathcal  I};\rset)$
and
$(\PP z)_j:=z(x_j)$ for $j=1,\dots,J$.
Finally, for $\ell\in\nset$ and for any function $q:\rset^{\,\ell}\rightarrow\rset$ and any
$w=(w^1,\dots,w^{\ell})\in(\gspace)^{\ell}$, we define
$q(w)\in\gspace$ by
$(q(w))_j:=q\left(w_j^1,\dots,w^{\ell}_j\right)$
for $j=0,\dots,J+1$.
%
%
\subsubsection{The Besse Relaxation Finite Difference method}\label{BRS_method}
\par
The Besse Relaxation Finite Difference (BRFD) method combines a standard finite
difference discetization in space with the Besse relaxation scheme in time
(cf. \cite{Besse1}). Its algorithm consists of the following steps:
\par\noindent\vskip0.2truecm\par\noindent
{\tt Step I}: Define $\UUU^0\in\dspace$ by
\begin{equation}\label{BRS_1}
\UUU^0:=\uu^0
\end{equation}
and then find $\UUU^{\half}\in\dspace$ such that
\begin{equation}\label{BRS_12}
\begin{split}
\tfrac{\UUU^{\half}-\UUU^0}{(\tau/2)}
=\Delta_{h}\left(\,\tfrac{\UUU^{\half}+\UUU^0}{2}\,\right)
+g\big(\uu^0\big)\otimes\left(\tfrac{\UUU^{\half}+\UUU^0}{2}\right)
+\PP\left[\tfrac{f(t^{\half},\cdot)+f(t_0,\cdot)}{2}\right].
\end{split}
\end{equation}
\par\noindent\vskip0.2truecm\par\noindent
{\tt Step II}: Define $\Phi^{\half}\in\gspace$ by
\begin{equation}\label{BRS_13}
\Phi^{\half}:=g(\UUU^{\half})
\end{equation}
and then find $\UUU^1\in\dspace$ such that
\begin{equation}\label{BRS_2}
\begin{split}
\tfrac{\UUU^1-\UUU^0}{\tau}=\Delta_{h}\left(\,\tfrac{\UUU^1+\UUU^0}{2}\,\right)
+\Phi^{\half}\otimes\left(\tfrac{\UUU^1+\UUU^0}{2}\right)
+\PP\left[\tfrac{f(t_{1},\cdot)+f(t_0,\cdot)}{2}\right].
\end{split}
\end{equation}
\par\noindent\vskip0.2truecm\par\noindent
{\tt Step III}: For $n=1,\dots,N-1$, first define $\Phi^{n+\half}\in\gspace$ by
\begin{equation}\label{BRS_3}
\Phi^{n+\half}:=2\,g(\UUU^n)-\Phi^{n-\half}
\end{equation}
and then find $\UUU^{n+1}\in\dspace$ such that
\begin{equation}\label{BRS_4}
\begin{split}
\tfrac{\UUU^{n+1}-\UUU^{n}}{\tau}=\Delta_h\left(\tfrac{\UUU^{n+1}+\UUU^{n}}{2}\right)
+\Phi^{n+\half}\otimes\left(\tfrac{\UUU^{n+1}+\UUU^{n}}{2}\right)
+\PP\left[\tfrac{f(t_{n+1},\cdot)+f(t_n,\cdot)}{2}\right].\\
\end{split}
\end{equation}
\par
Obviously, the numerical method above requires, at each time step, the solution of
a tridiagonal linear system of algebraic equations.
%
\subsection{An overview of the paper}
\par
In the error analysis of the (BRFD) method, we face the locally Lipschitz nonlinearity of the
problem by introducing the (MBRFD) scheme (see Section~\ref{The_MBRFD}),
which follows from the (BRFD) method after molifying properly the terms with nonlinear
structure (cf. \cite{Akrivis1}, \cite{Georgios1}, \cite{KarMak}).
The (MBRFD) approximations depend on a parameter
$\delta>0$ and have the following key property: when their discrete $L^{\infty}$-norm is bounded by $\delta$,
then they are also (BRFD) approximations, because, in that case, the molifier (see \eqref{ni_defin})
acts as an indentity.
Assuming that $\delta$ is large enough and $\tau$ is sufficiently small, for the non
computable (ΜBRFD) approximations, first we show that are well-defined
(see Proposition~\ref{BR_Exist}), and then we establish an optimal, second order
error estimate in the discrete $H^1$-norm (see Theorem~\ref{BR_CB_Conv}).
Letting $h$ and $\tau$ be sufficiently small (see \eqref{BR_Xmesh})
and applying a discrete Sobolev inequality (see \eqref{LH1}), the latter
convergence result implies that the discrete $L^{\infty}-$norm of the (MBRFD) approximations
are lower than $\delta$ and thus they, also, are (BRFD) approximations.
Finally, we are show that the (BRFD) approximations are unique
and hence inherit the convergence properties of the (MBRFD) scheme (see Theorem~\ref{DR_Final}),
i.e. that there exist constants $C_1$ and $C_2$, independent of $\tau$ and $h$, such that
\begin{equation*}
\big|\,\UUU^{\half}-\PP[\uu(t^{\half},\cdot)]\,\big|_{1,h}
\leq\,C_1\,(\tau^2+\tau^{\half}\,h^2)
\end{equation*}
and
\begin{equation*}
\max_{0\leq n\leq {\ssy N}}\left[\,
\big|\,\Phi^{n+\half}-{\sf I}_h[g(u(t^{n+\half},\cdot))]\,\big|_{1,h}
+\big|\,U^n-\PP[u(t_n,\cdot)]\,\big|_{1,h}\,\right]\leq\,C_2\,(\tau^2+h^2),
\end{equation*}
where $|\cdot|_{1,h}$ is a discrete $H^1-$norm which is stronger than the discrete $L^{\infty}-$norm.
%
%
\par
At every time-step, the (BRFD) method computes first an approximation
of $g(u)$ at the midpoint of the current time interval (see \eqref{BRS_13} and \eqref{BRS_3})
and then an approximation of $u$ at the next time node (see \eqref{BRS_2} and \eqref{BRS_4}).
%
%
However, the computation of the approximations of $g(u)$ at the midpoints is a simple
postprocessing procedure and has no obvious discrete dynamic structure. The stability argument
we employ is based first on taking a discrete derivative of the error equation that corresponds to
\eqref{BRS_3} (see \eqref{BR_CB_Gat7}) and then on including the discrete $L^2$
and discrete $H^1$ norm of the time increment of the error in the stability norm (see
\eqref{BR_Zerror} and \eqref{BR_NES_10_a}).
%
\par
We close this section by giving a brief overview of the paper.
In Section~\ref{Section2}, we introduce additional notation
and provide a series of auxiliary results.
Section~\ref{Section3} is dedicated to the estimation of
several type of consistency errors and of the approximation
error of a discrete elliptic projection.
In Section~\ref{Section4}, we define a modified version of the (BRFD) method,
and then analyze its convergence properties and arrive at a set of conditions
that ensure the well-posedness and convergence of the (BRFD) method.
%
%
%
\section{Preliminaries}\label{Section2}
%
%
%
Let us introduce another discrete space by
$\sspace:=\left\{\,(z_j)_{j=0}^{\ssy J}:\,\,z_j\in\rset,\quad j=0,\dots,J\right\}$
and the discrete space derivative operator
$\delta_h:\gspace\rightarrow\sspace$ by
\begin{equation*}
\delta_hv_j:=\tfrac{v_{j+1}-v_j}{h},\quad j=0,\dots,J,
\quad\forall\,v\in\gspace.
\end{equation*}
We define on $\sspace$ an inner product $(\!\!(\cdot,\cdot)\!\!)_{0,h}$ by
$(\!\!(z,v)\!\!)_{0,h}:=h\,\sum_{j=0}^{\ssy J}z_j\,v_j$ for $z,v\in\sspace$,
and we will denote by $|\!|\!|\cdot|\!|\!|_{0,h}$ the corresponding norm,
i.e. $|\!|\!|z|\!|\!|_{0,h}:=\left[(\!\!(z,z)\!\!)_{0,h}\right]^{\ssy 1/2}$ for $z\in\sspace$.
Also, we define a discrete maximum norm $|\!|\!|\cdot|\!|\!|_{\infty,h}$ on $\sspace$
by $|\!|\!|v|\!|\!|_{\infty,h}:=\max_{0\leq{j}\leq{\ssy J}}|v_j|$ for $v\in\sspace$.
\par
We provide $\dspace$ with the discrete inner product
$(\cdot,\cdot)_{0,h}$ given by $(v,z)_{0,h}:=h\,\sum_{j=1}^{\ssy J}v_j\,z_j$
for $v,z\in\dspace$, and we shall denote by $\|\cdot\|_{0,h}$ its
induced norm, i.e.
$\|v\|_{0,h}:=\left[(v,v)_{0,h}\right]^{\ssy 1/2}$ for $v\in\dspace$.
Also, we equip $\gspace$ with
a discrete $L^{\infty}$-norm $|\cdot|_{\infty,h}$
defined by $|w|_{\infty,h}:=\max_{0\leq j\leq{J+1}}|w_j|$ for
$w\in\gspace$, and with a discrete $H^1$-seminorm $|\cdot|_{1,h}$
given by $|w|_{1,h}:=|\!|\!|\delta_hw|\!|\!|_{0,h}$ for $w\in\gspace$.
It is easily seen that $|\cdot|_{1,h}$ becomes a norm when it is restricted on $\dspace$
and satisfies the following useful inequalities:
\begin{gather}
|v|_{\infty,h}\leq\,{\sf L}^{\ssy 1/2}\,|v|_{1,h},\label{LH1}\\
\|v\|_{0,h}\leq\,{\sf L}\,|v|_{1,h}\label{dpoincare}
\end{gather}
for $v\in\dspace$.
%
In the sequel, we present a series of auxiliary results that they will be in often use in the rest of the work.
\begin{lem}
For all $v,z\in\dspace$ it holds that
\begin{gather}
(\Delta_hv,z)_{0,h}=-(\!\!(\delta_hv,\delta_hz)\!\!)_{0,h}=(v,\Delta_hz)_{0,h},
\label{NewEra1}\\
(\Delta_hv,v)_h=-|v|^2_{1,h}.\label{NewEra2}
\end{gather}
\end{lem}
%
%
%
%
\begin{proof}
Let $v,z\in\dspace$. First, we establish \eqref{NewEra1}
proceeding as follows:
\begin{equation*}
\begin{split}
(\Delta_hv,z)_{0,h}=&\,\sum_{j=1}^{\ssy J}\left[\,(\delta_hv)_{j}-(\delta_hv)_{j-1}\right]\,z_j
=\sum_{j=0}^{\ssy J}(\delta_hv)_{j}\,z_j-\sum_{j=0}^{\ssy J}(\delta_hv)_j\,z_{j+1}
=-(\!\!(\delta_hv,\delta_hz)\!\!)_{0,h}.\\
\end{split}
\end{equation*}
Then, we set $z=v$ in \eqref{NewEra1} to get \eqref{NewEra2}.
\end{proof}
%
%
%
%
%
%
\begin{lem}\label{Lemma_D2}
Let ${\mathfrak g}\in C_b^2(\rset;\rset)$. Then, for $v,w\in\dspace$, it holds that
\begin{equation}\label{Gprop0}
|\mfg(v)-\mfg(w)|_{1,h}\leq
\mfg'_{\infty}\,|v-w|_{1,h}
+\mfg''_{\infty}
\,|\!|\!|\delta_hw|\!|\!|_{\infty,h}\,\|v-w\|_{0,h}
\end{equation}
where  $\mfg'_{\infty}:=\sup_{\ssy\rset}|\mfg'|$ and  $\mfg''_{\infty}:=\sup_{\ssy\rset}|\mfg''|$.
\end{lem}
%
%
%
%
%
\begin{proof}
Let $v,w\in\dspace$. First, we define ${\mathfrak a}^s$, ${\mathfrak b}^s\in\sspace$
by ${\mathfrak a}^s_j:=s\,v_{j+1}+(1-s)\,v_j$
and ${\mathfrak b}^s_j:=s\,w_{j+1}+(1-s)\,w_j$ for
$j=0,\dots,J$ and $s\in[0,1]$.
Then,  we use the mean value theorem, to conclude that
\begin{equation}\label{AGP_1}
\delta_h(\mfg(v)-\mfg(w))={\mathcal L}^{\ssy A}+{\mathcal L}^{\ssy B}
\end{equation}
where ${\mathcal L}^{\ssy A},{\mathcal L}^{\ssy B}\in\sspace$ given by
${\mathcal L}^{\ssy A}_j:=(\delta_h(v-w))_j\,\int_0^1\mfg'({\mathfrak a}^s_j)\;ds$
and ${\mathcal L}^{\ssy B}_j:=\delta_hw_j\,\int_0^1\left[\mfg'({\mathfrak a}^s_j)
-\mfg'({\mathfrak b}^s_j)\right]\;ds$
for $j=0,\dots,J$.
Observing that
\begin{equation*}
\left|{\mathcal L}^{\ssy A}_j\right|\leq
\sup_{\ssy\rset}|\mfg'|\,|(\delta_h(v-w))_j|,
\quad j=0,\dots,J,
\end{equation*}
and
\begin{equation*}
\begin{split}
\left|{\mathcal L}^{\ssy B}_j\right|\leq&\,|(\delta_hw)_j|\,
\sup_{\ssy\rset}|\mfg''|\left|\int_0^1\left[\,s(v_{j+1}-w_{j+1})+(1-s)\,(v_j-w_j)\,\right]\;ds\right|\\
\leq&\,\tfrac{1}{2}\,|(\delta_hw)_j|\,\sup_{\ssy\rset}|\mfg''|\,
\left(|v_{j+1}-w_{j+1}|+|v_j-w_j|\right),
\quad j=0,\dots,J,\\
\end{split}
\end{equation*}
we, easily, arrive at
\begin{gather}
|\!|\!|{\mathcal L}^{\ssy A}|\!|\!|_{0,h}\leq\,\sup_{\ssy\rset}|\mfg'|\,
|\!|\!|\delta_h(v-w)|\!|\!|_{0,h},\label{AGP_2}\\
|\!|\!|{\mathcal L}^{\ssy B}|\!|\!|_{0,h}\leq\,|\!|\!|\delta_hw|\!|\!|_{\infty,h}
\,\sup_{\ssy\rset}|\mfg''|\,\|v-w\|_{0,h}.\label{AGP_3}
\end{gather}
Thus, \eqref{Gprop0} follows as a simple consequence of \eqref{AGP_1}, \eqref{AGP_2} and \eqref{AGP_3}.
\end{proof}
%
%
%
%
\begin{lem}\label{Lemma_D3}
Let $\mfg\in C^3_b(\rset;\rset)$. Then, for $v^a,v^b,z^a,z^b\in\dspace$, it holds that
\begin{equation}\label{OPAP_1A}
\begin{split}
\|\mfg(v^a)-\mfg(v^b)-\mfg(z^a)+\mfg(z^b)\|_{0,h}
\leq&\,\mfg''_{\infty}\,\,|z^a-z^b|_{\infty,h}\,\|v^b-z^b\|_{0,h}\\
&+\left(\mfg'_{\infty}+\mfg''_{\infty}\,\,|z^a-z^b|_{\infty,h}\right)
\,\|v^a-v^b-z^a+z^b\|_{0,h}\\
\end{split}
\end{equation}
and
\begin{equation}\label{OPAP_1B}
\begin{split}
|\mfg(v^a)-\mfg(v^b)-\mfg(z^a)+\mfg(z^b)|_{1,h}
\leq&\,{\mathcal F}^{\ssy A}(v^a,v^b)\,
|v^a-v^b-z^a+z^b|_{1,h}\\
&+{\mathcal F}^{\ssy B}(z^a,z^b)\,
\,\left(\,\|v^a-v^b-z^a+z^b\|_{0,h}+\|v^b-z^b\|_{0,h}\,\right)\\
&+{\mathcal F}^{\ssy C}(z^a,z^b)\,
\,\left(\,|v^a-v^b-z^a+z^b|_{1,h}+|v^b-z^b|_{1,h}\,\right),\\
\end{split}
\end{equation}
where $\mfg'_{\infty}:=\sup_{\ssy\rset}|\mfg'|$, $\mfg''_{\infty}:=\sup_{\ssy\rset}|\mfg''|$,
\begin{equation*}
\begin{split}
{\mathcal F}^{\ssy A}(v^a,v^b):=&\,{\mathfrak g}'_{\infty}
+\tfrac{{\sf L}^{1/2}}{2}\,{\mathfrak g}''_{\infty}\,(\,|v^a|_{1,h}+|v^b|_{1,h}\,),\\
{\mathcal F}^{\ssy B}(z^a,z^b):=&\,\mfg''_{\infty}\,\,|\!|\!|\delta_h(z^a-z^b)|\!|\!|_{\infty,h},\\
{\mathcal F}^{\ssy C}(z^a,z^b):=&\,|z^a-z^b|_{1,h}\,\left[\mfg''_{\infty}+{\sf L}\,{\mathfrak g}'''_{\infty}
\,\left(|\!|\!|\delta_hz^a|\!|\!|_{\infty,h}+|\!|\!|\delta_hz^b|\!|\!|_{\infty,h}\right)\right]\\
\end{split}
\end{equation*}
and $\mfg'''_{\infty}:=\sup_{\ssy\rset}|\mfg'''|$.
\end{lem}
%
%
%
%
\begin{proof}
Let $v^a,v^b,z^a,z^b\in\dspace$. We simplify the notation, first, by defining
${\mathfrak a}^s$, ${\mathfrak b}^s\in\dspace$ by
${\mathfrak a}^s:=s\,v^a+(1-s)\,v^b$ and ${\mathfrak b}^s:=s\,z^a+(1-s)\,z^b$
for $s\in[0,1]$, and then, by introducing ${\mathfrak f}\in\gspace $ by
${\mathfrak f}:=\int_0^1\mfg'({\mathfrak a}^s)\;ds$
and
${\mathfrak t}\in\dspace$ by
${\mathfrak t}:=\int_0^1\left[{\mathfrak g}'({\mathfrak a}^s)
-{\mathfrak g}'({\mathfrak b}^s)\right]\;ds$. Also, we set $e^a:=v^a-z^a$ and
$e^b:=v^b-z^b$.
%
%
\par\noindent\vskip0.3truecm\par\noindent
\boxed{\sf Part\,\,\, I.} First, we use the definition of ${\mathfrak f}$ and
the mean value theorem, to get
\begin{equation}\label{OPAP_2}
|{\mathfrak f}|_{\infty,h}\leq\,\mfg'_{\infty}
\end{equation}
and
\begin{equation*}
\begin{split}
\left|\delta_h{\mathfrak f}_j\right|\leq&\,\tfrac{1}{h}\,\int_0^1|\mfg'({\mathfrak a}^s_{j+1})
-\mfg'({\mathfrak a}^s_j)|\;ds\\
\leq&\,\mfg''_{\infty}\,
\int_0^1\left|s\,\delta_hv^a_j+(1-s)\,\delta_hv^b_j\,\right|\;ds\\
\leq&\,\tfrac{1}{2}\,\mfg''_{\infty}
\,\left(\,|\delta_hv^a_j|+|\delta_hv^b_j|\,\right),
\quad j=0,\dots,J,\\
\end{split}
\end{equation*}
which, obviously, yields
\begin{equation}\label{OPAP_3}
|{\mathfrak f}|_{1,h}\leq\,\tfrac{1}{2}\,\mfg''_{\infty}
\,\left(|v^a|_{1,h}+|v^b|_{1,h}\right).
\end{equation}
Next, we use the definition of ${\mathfrak t}$ and the mean value theorem,
to obtain
\begin{equation*}
\begin{split}
|{\mathfrak t}_j|\leq&\,\mfg''_{\infty}\int_0^1|{\mathfrak a}_j^s-{\mathfrak b}_j^s|
\;ds\\
\leq&\,\mfg''_{\infty}\,\,
\int_0^1|s\,(v^a_j-v^b_j-z_j^a+z_j^b)
+(v_j^b-z_j^b)|\;ds\\
\leq&\,\mfg''_{\infty}\,\,\left(\,|v^a_j-v^b_j-z_j^a+z_j^b|
+|v_j^b-z_j^b|\,\right),\quad j=1,\dots,J,\\
\end{split}
\end{equation*}
which, leads to
\begin{equation}\label{OPAP_4}
\|{\mathfrak t}\|_{0,h}\leq\,\mfg''_{\infty}\,
\left(\,\|e^a-e^b\|_{0,h}+\|e^b\|_{0,h}\,\right).
\end{equation}
Finally, for $s\in[0,1]$, we apply \eqref{Gprop0} and \eqref{dpoincare},
to arrive at
\begin{equation}\label{OPAP_5}
\begin{split}
|{\mathfrak g}'({\mathfrak a}^s)-{\mathfrak g}'({\mathfrak b}^s)|_{1,h}
\leq&\,\mfg''_{\infty}\,|{\mathfrak a}^s-{\mathfrak b}^s|_{1,h}
+{\mathfrak g}'''_{\infty}
\,|\!|\!|\delta_h{\mathfrak b}^s|\!|\!|_{\infty,h}\,\|{\mathfrak a}^s-{\mathfrak b}^s\|_{0,h}\\
\leq&\,\big(\,\mfg''_{\infty}+{\sf L}\,{\mathfrak g}'''_{\infty}
\,|\!|\!|\delta_h{\mathfrak b}^s|\!|\!|_{\infty,h}\,\big)\,|{\mathfrak a}^s-{\mathfrak b}^s|_{1,h}\\
\leq&\,\big(\,\mfg''_{\infty}+{\sf L}\,{\mathfrak g}'''_{\infty}
\,|\!|\!|\delta_h{\mathfrak b}^s|\!|\!|_{\infty,h}\,\big)
\,\left(\,|e^a-e^b|_{1,h}+\,|e^b|_{1,h}\,\right).\\
\end{split}
\end{equation}
Observing that $\delta_h{\mathfrak t}=
\int_0^1\delta_h\left[{\mathfrak g}'({\mathfrak a}^s)
-{\mathfrak g}'({\mathfrak b}^s)\right]\;ds$
and using \eqref{OPAP_5} we have
\begin{equation}\label{OPAP_6}
\begin{split}
|{\mathfrak t}|_{1,h}\leq&\int_0^1|{\mathfrak g}'({\mathfrak a}^s)
-{\mathfrak g}'({\mathfrak b}^s)|_{1,h}\;ds\\
\leq&\,\big[\mfg''_{\infty}+{\sf L}\,{\mathfrak g}'''_{\infty}
\,\left(|\!|\!|\delta_hz^a|\!|\!|_{\infty,h}+|\!|\!|\delta_hz^b|\!|\!|_{\infty,h}\right)\big]
\,\left(\,|e^a-e^b|_{1,h}+|e^b|_{1,h}\,\right).\\
\end{split}
\end{equation}
%
%
%
\par\noindent\vskip0.3truecm\par\noindent
\boxed{\sf Part\,\,\, II.} Using the mean value theorem, we obtain
\begin{equation}\label{OPAP_7}
{\mathfrak g}(v^a)-{\mathfrak g}(v^b)-{\mathfrak g}(z^a)+{\mathfrak g}(z^b)
={\mathfrak L}^{\ssy A}+{\mathfrak L}^{\ssy B},
\end{equation}
where ${\mathfrak L}^{\ssy A}$, ${\mathfrak L}^{\ssy B}\in\dspace$ are defined by
${\mathfrak L}^{\ssy A}:=(v^a-v^b-z^a+z^b)\otimes{\mathfrak f}$ and
${\mathfrak L}^{\ssy B}:=(z^a-z^b)\otimes{\mathfrak t}$. Thus, using
\eqref{OPAP_2} and \eqref{OPAP_4}, we have
\begin{equation}\label{OPAP_8}\
\begin{split}
\|{\mathfrak L}^{\ssy A}\|_{0,h}\leq&\,\mfg'_{\infty}\,\|e^a-e^b\|_{0,h},\\
\|{\mathfrak L}^{\ssy B}\|_{0,h}\leq&\,\mfg''_{\infty}
\,|z^a-z^b|_{\infty,h}\,\left(\|e^a-e^b\|_{0,h}+\|e^b\|_{0,h}\right).\\
\end{split}
\end{equation}
The desired inequality \eqref{OPAP_1A} follows, easily, as a simple outcome
of \eqref{OPAP_7} and \eqref{OPAP_8}.
%
%
%
%
\par\noindent\vskip0.3truecm\par\noindent
\boxed{\sf Part\,\,\, III.} For the discrete derivative of ${\mathfrak L}^{\ssy A}$
and ${\mathfrak L}^{\ssy B}$, we, easily, obtain the following formulas:
\begin{equation*}
\begin{split}
(\delta_h{\mathfrak L}^{\ssy A})_{j}=&\,\delta_h(v^a-v^b-z^a+z^b)_j\,{\mathfrak f}_{j+1}
+(v^a_{j}-v^b_{j}-z_{j}^a+z_{j}^b)\,(\delta_h{\mathfrak f})_j,\\
(\delta_h{\mathfrak L}^{\ssy B})_{j}=&\,\delta_h(z^a-z^b)_j\,{\mathfrak t}_{j+1}
+(z^a-z^b)_j\,(\delta_h{\mathfrak t})_j\\
\end{split}
\end{equation*}
for $j=0,\dots,J$, which yield
\begin{equation}\label{OPAP_9}
\begin{split}
|{\mathfrak L}^{\ssy A}|_{1,h}\leq&\,|e^a-e^b|_{1,h}\,|{\mathfrak f}|_{\infty,h}
+|e^a-e^b|_{\infty,h}\,|{\mathfrak f}|_{1,h},\\
|{\mathfrak L}^{\ssy B}|_{1,h}\leq&\,|\!|\!|\delta_h(z^a-z^b)|\!|\!|_{\infty,h}\,\|{\mathfrak t}\|_{0,h}
+|z^a-z^b|_{\infty,h}\,|{\mathfrak t}|_{1,h}.\\
\end{split}
\end{equation}
Using \eqref{OPAP_9}, \eqref{LH1}, \eqref{OPAP_2} and \eqref{OPAP_3}, we have
\begin{equation}\label{OPAP_10}
|{\mathfrak L}^{\ssy A}|_{1,h}\leq\,\left[{\mathfrak g}'_{\infty}
+\tfrac{{\sf L}^{1/2}}{2}\,{\mathfrak g}''_{\infty}\,(\,|v^a|_{1,h}+|v^b|_{1,h}\,)\right]
\,|e^a-e^b|_{1,h}.
\end{equation}
Combining \eqref{OPAP_9}, \eqref{OPAP_4}, \eqref{OPAP_6} and \eqref{LH1}, we arrive
at
\begin{equation}\label{OPAP_11}
\begin{split}
|{\mathfrak L}^{\ssy B}|_{1,h}\leq&\,\mfg''_{\infty}\,|\!|\!|\delta_h(z^a-z^b)|\!|\!|_{\infty,h}\,
\left(\,\|e^a-e^b\|_{0,h}+\|e^b\|_{0,h}\,\right)\\
&\,+|z^a-z^b|_{1,h}\,\big[\mfg''_{\infty}+{\sf L}\,{\mathfrak g}'''_{\infty}
\,\left(|\!|\!|\delta_hz^a|\!|\!|_{\infty,h}+|\!|\!|\delta_hz^b|\!|\!|_{\infty,h}\right)\big]
\,\left(\,|e^a-e^b|_{1,h}+|e^b|_{1,h}\,\right).\\
\end{split}
\end{equation}
Finally, \eqref{OPAP_1B} follows, easily, in view of \eqref{OPAP_7}, \eqref{OPAP_10} and
\eqref{OPAP_11}.
\end{proof}
%
%
%
\section{Consistency Errors}\label{Section3}
%
%
To simplify the notation, we set $t^{\quarter}:=\tfrac{\tau}{4}$,
$\uu^{\frac{1}{4}}:={\sf I}_h[\uu(t^{\quarter},\cdot)]$,
$\uu^{n}:={\sf I}_h[\uu(t_{n},\cdot)]$
for $n=0,\dots,N$, and $\uu^{n+\half}:={\sf I}_h[\uu(t^{n+\half},\cdot)]$ for $n=0,\dots,N-1$.
In view of the Dirichlet boundary conditions \eqref{PSL_b} and the compatibility conditions \eqref{PSL_d},
it holds that $\uu^{\frac{1}{4}}\in\dspace$, $\uu^n\in\dspace$ for $n=0,\dots,N$
and $\uu^{n+\half}\in\dspace$ for $n=0,\dots,N-1$.
\subsection{Time consistency error at the nodes}
Let ${\sf r}^{\frac{1}{4}}\in\gspace$ be defined by
\begin{equation}\label{NORAD_00}
\tfrac{\uu^{\half}-\uu^0}{(\tau/2)}={\sf I}_h\left[
\tfrac{\uu_{xx}(t^{\half},\cdot)+\uu_{xx}(t_0,\cdot)}{2}\right]
+g(\uu^0)\otimes\left(\tfrac{\uu^{\half}+\uu^0}{2}\right)
+{\sf I}_h\left[\tfrac{f(t^{\half},\cdot)+f(t_0,\cdot)}{2}\right]+{\sf r}^{\quarter}
\end{equation}
%
%
and let ${\sf r}^{n+\half}\in\gspace$ be specified by
\begin{equation}\label{NORAD_02}
\tfrac{\uu^{n+1}-\uu^n}{\tau}={\sf I}_h\left[\tfrac{\uu_{xx}(t_{n+1},\cdot)+\uu_{xx}(t_n,\cdot)}{2}\right]
+g(\uu^{n+\half})\otimes\left(\tfrac{\uu^{n+1}+\uu^n}{2}\right)
+{\sf I}_h\left[\tfrac{f(t_{n+1},\cdot)+f(t_n,\cdot)}{2}\right]+{\sf r}^{n+\half}
\end{equation}
for $n=0,\dots,N-1$.
Assuming that the solution $\uu$ is smooth enough on $[0,T]\times{\mathcal I}$, and using
\eqref{PSL_d} and the Dirichlet boundary conditions \eqref{PSL_b}, we conclude
that $u_{xx}(t,x)=-f(t,x)$ for $t\in[0,T]$ and $x\in\{x_a,x_b\}$.
Thus, we have ${\sf r}^{\frac{1}{4}}\in\dspace$ and ${\sf r}^{n+\half}\in\dspace$ for $n=0,\dots,N-1$.
\par
Substracting \eqref{PSL_a} with $(t,x)=(t^{\frac{1}{4}},x_j)$ from \eqref{NORAD_00},
and \eqref{PSL_a} with $(t,x)=(t^{n+\half},x_j)$ from \eqref{NORAD_02}, we get
\begin{equation}\label{USNORTH_01}
\begin{split}
{\sf r}^{\quarter}=&\,{\sf r}^{\quarter}_{\ssy A}-{\sf r}^{\quarter}_{\ssy B}
-{\sf r}^{\quarter}_{\ssy C}-{\sf r}^{\quarter}_{\ssy D},\quad
{\sf r}^{n+\half}={\sf r}^{n+\half}_{\ssy A}-{\sf r}^{n+\half}_{\ssy B}
-{\sf r}^{n+\half}_{\ssy C}-{\sf r}^{n+\half}_{\ssy D},\quad n=0,\dots,N-1,
\end{split}
\end{equation}
where
${\sf r}^{\quarter}_{\ssy A},{\sf r}^{\quarter}_{\ssy C},
{\sf r}^{n+\half}_{\ssy A},{\sf r}^{n+\half}_{\ssy C}\in\dspace$
and ${\sf r}^{\quarter}_{\ssy B},{\sf r}^{\quarter}_{\ssy D},
{\sf r}^{n+\half}_{\ssy B},{\sf r}^{n+\half}_{\ssy D}\in\gspace$ be defined by
\begin{gather*}
{\sf r}^{n+\half}_{\ssy A}:=\tfrac{\uu^{n+1}-\uu^n}{\tau}
-{\sf I}_h\big[\uu_t\big(t^{n+\half},\cdot\big)\big],\quad
{\sf r}^{n+\half}_{\ssy B}:={\sf I}_h\left[\tfrac{\uu_{xx}(t_{n+1},\cdot)+\uu_{xx}(t_n,\cdot)}{2}
-u_{xx}(t^{n+\half},\cdot)\right],\\
{\sf r}^{n+\half}_{\ssy C}:=g(\uu^{n+\half})\otimes\left[\tfrac{\uu^{n+1}+\uu^n}{2}
-\uu^{n+\half}\right],
\quad
{\sf r}^{n+\half}_{\ssy D}:={\sf I}_h\left[\tfrac{f(t_{n+1},\cdot)+f(t_n,\cdot)}{2}
-f(t^{n+\half},\cdot)\right]\\
\end{gather*}
and
\begin{gather*}
{\sf r}^{\quarter}_{\ssy A}:=\tfrac{\uu^{\half}-\uu^0}{(\tau/2)}
-{\sf I}_h\big[\uu_t\big(t^{\quarter},\cdot\big)\big],\quad
{\sf r}^{\quarter}_{\ssy B}:=\PP\left[\tfrac{\uu_{xx}(t^{\half},\cdot)+\uu_{xx}(t_0,\cdot)}{2}
-u_{xx}(t^{\quarter},\cdot)\right],\\
{\sf r}^{\quarter}_{\ssy C}:=-\left[g(\uu^{\quarter})-g(\uu^0)\right]
\otimes\uu^{\quarter}
+g(\uu^0)\otimes\left[\tfrac{\uu^{\half}+\uu^0}{2}-\uu^{\quarter}\right],\\
{\sf r}^{\quarter}_{\ssy D}:={\sf I}_h\left[\tfrac{f(t^{\half},\cdot)+f(t_0,\cdot)}{2}
-f(t^{\quarter},\cdot)\right].\\
\end{gather*}
Applying the Taylor formula we obtain
\begin{equation}\label{USNORTH_02A}
\begin{split}
({\sf r}_{\ssy A}^{n+\half})_j=&\,\tfrac{\tau^2}{2}\,\int_0^{\half}\left[\,s^2\,\uu_{ttt}(t_n+s\,\tau,x_j)
+(\tfrac{1}{2}-s)^2\,\uu_{ttt}(t^{n+\half}+s\,\tau,x_j)\,\right]\;ds,\\
({\sf r}^{n+\half}_{\ssy C})_j=&\,\tfrac{g(\uu(t^{n+\half},x_j))}{2}\,\tau^2
\,\int_0^{\frac{1}{2}}\left[\,s\,\uu_{tt}(t_n+s\,\tau,x_j)
+(\tfrac{1}{2}-s)\,\uu_{tt}(t^{n+\half}+s\,\tau,x_j)\,\right]\;ds,\\
({\sf r}_{\ssy B}^{n+\half})_j=&\,
\tfrac{\tau^2}{2}\,\int_0^{\frac{1}{2}}\left[\,s\,\uu_{xxtt}(t_n+s\,\tau,x_j)
+(\tfrac{1}{2}-s)\,\uu_{xxtt}(t^{n+\half}+s\,\tau,x_j)\,\right]\;ds,\\
({\sf r}_{\ssy D}^{n+\half})_j=&\,
\tfrac{\tau^2}{2}\,\int_0^{\frac{1}{2}}\left[\,s\,f_{tt}(t_n+\tau\,s,x_j)
+(\tfrac{1}{2}-s)\,f_{tt}(t^{n+\half}+\tau\,s,x_j)\,\right]\;ds\\
\end{split}
\end{equation}
for $j=0,\dots,J+1$ and $n=0,\dots,N-1$, and
\begin{equation}\label{USNORTH_04A}
\begin{split}
({\sf r}_{\ssy A}^{\quarter})_j=&\,\tfrac{\tau^2}{2}\,\int_0^{\quarter}\left[\,s^2\,\uu_{ttt}(s\,\tau,x_j)
+(\tfrac{1}{4}-s)^2\,\uu_{ttt}(t^{\quarter}+s\,\tau,x_j)\,\right]\;ds,\\
({\sf r}^{\quarter}_{\ssy C})_j=&\,-\uu(t^{\quarter},x_j)\,\tau\,
\int_{0}^{\ssy\frac{1}{4}}g'(u(s\,\tau,x_j))\,u_t(s\,\tau,x_j)\;ds\\
&\,+\tfrac{g(\uu_0(x_j))}{2}\,\tau^2\,\int_0^{\ssy\frac{1}{4}}\left[\,s\,\uu_{tt}(s\,\tau,x_j)
+(\tfrac{1}{4}-s)\,\uu_{tt}(t^{\quarter}+s\,\tau,x_j)\,\right]\;ds,\\
({\sf r}_{\ssy B}^{\quarter})_j=&\,
\tfrac{\tau^2}{2}\,\int_0^{\ssy\frac{1}{4}}\left[\,s\,\uu_{xxtt}(s\,\tau,x_j)
+(\tfrac{1}{4}-s)\,\uu_{xxtt}(t^{\quarter}+s\,\tau,x_j)\,\right]\;ds,\\
({\sf r}_{\ssy D}^{\quarter})_j=&\,
\tfrac{\tau^2}{2}\,\int_0^{\frac{1}{4}}\left[\,s\,f_{tt}(t_n+\tau\,s,x_j)
+(\tfrac{1}{4}-s)\,f_{tt}(t^{n+\half}+\tau\,s,x_j)\,\right]\;ds\\
\end{split}
\end{equation}
for $j=0,\dots,J+1$. Then, from
\eqref{USNORTH_01}, \eqref{USNORTH_02A} and \eqref{USNORTH_04A},
we arrive at
\begin{gather}
\|{\sf r}_{\ssy A}^{\quarter}\|_{0,h}+\|{\sf r}^{\quarter}_{\ssy B}\|_{0,h}
+\|{\sf r}^{\quarter}_{\ssy D}\|_{0,h}
+\max_{0\leq{n}\leq{\ssy N-1}}
\|{\sf r}^{n+\half}\|_{0,h}\leq\,{\widehat{\sf C}}_{1,1}\,\tau^2,\label{SHELL_01a}\\
\|{\sf r}_{\ssy C}^{\quarter}\|_{0,h}\leq\,{\widehat{\sf C}}_{1,2}\,\tau\label{SHELL_01b}
\end{gather}
and
\begin{gather}
\max_{0\leq{n}\leq{\ssy N-1}}|{\sf r}^{n+\half}|_{1,h}
\leq\,{\widehat{\sf C}}_{1,3}\,\tau^2,\label{SHELL_02a}\\
|{\sf r}_{\ssy C}^{\quarter}|_{1,h}\leq\,{\widehat{\sf C}}_{1,4}\,\tau.\label{SHELL_02b}
\end{gather}
%
%
%
\subsection{Space consistency error}
%
%
Also, let ${\sf s}^{\quarter}\in\dspace$ be defined by
\begin{equation}\label{NORAD_11}
\tfrac{\uu^{\half}-\uu^0}{(\tau/2)}=\Delta_h\left(\tfrac{\uu^{\half}+\uu^0}{2}\right)
+g(\uu^0)\otimes\left(\tfrac{\uu^{\half}+\uu^0}{2}\right)
+\PP\left[\tfrac{f(t^{\half},\cdot)+f(t_0,\cdot)}{2}\right]+{\sf s}^{\quarter}
\end{equation}
and, for $n=0,\dots,N-1$, let ${\sf s}^{n+\half}\in\dspace$ be given by
\begin{equation}\label{NORAD_12}
\tfrac{\uu^{n+1}-\uu^n}{\tau}=\Delta_h\left(\tfrac{\uu^{n+1}+\uu^n}{2}\right)
+g(\uu^{n+\half})\otimes\left(\tfrac{\uu^{n+1}+\uu^n}{2}\right)
+\PP\left[f\tfrac{f(t_{n+1},\cdot)+f(t_n,\cdot)}{2}\right]+{\sf s}^{n+\half}.
\end{equation}
Subtracting \eqref{NORAD_11} from \eqref{NORAD_00} and \eqref{NORAD_12} from
\eqref{NORAD_02}, we obtain
\begin{equation}\label{USNORTH_11}
\begin{split}
{\sf r}^{\quarter}-{\sf s}^{\quarter}=&\,\PP\left[\tfrac{\uu_{xx}(t^{\half},\cdot)+\uu_{xx}(t_0,\cdot)}{2}\right]
-\Delta_h\left(\tfrac{\uu^{\half}+\uu^0}{2}\right),\\
{\sf r}^{n+\half}-{\sf s}^{n+\half}=&\,\PP\left[\tfrac{\uu_{xx}(t_{n+1},\cdot)+\uu_{xx}(t_n,\cdot)}{2}\right]
-\Delta_h\left(\tfrac{\uu^{n+1}+\uu^n}{2}\right),\quad n=0,\dots,N-1.
\end{split}
\end{equation}
The use of the Taylor formula yields
\begin{equation*}\label{USARCTIC}
\begin{split}
\left(\PP\left[\uu_{xx}(t,\cdot)\right]
-\Delta_h\left({\sf I}_h[u(t,\cdot)]\right)\right)_j=&\,\tfrac{h^2}{6}
\,\int_0^1(1-y)^3\,\uu_{xxxx}(t,x_j+h\,y)\;dy\\
&+\tfrac{h^2}{6}
\,\int_0^1y^3\,\uu_{xxxx}(t,x_{j-1}+h\,y)\;dy,\\
\end{split}
\end{equation*}
for $j=1,\dots,J$ and $t\in[0,T]$, which along with \eqref{USNORTH_11}
yields
\begin{equation}\label{SHELL_11}
\|{\sf s}^{\quarter}-{\sf r}^{\quarter}\|_{0,h}
+\max_{0\leq{n}\leq{\ssy N-1}}\|{\sf s}^{n+\half}-{\sf r}^{n+\half}\|_{0,h}
\leq\,{\widehat C}_{2}\,h^2.
\end{equation}
%
%
%
%
\subsection{Time consistency error at the intermediate nodes}
%
%
%
%
For $n=1,\dots,N-1$, let ${\sf r}^n\in\dspace$
be determined by
\begin{equation}\label{NORAD_22}
\tfrac{g(\uu^{n+\half})+g(\uu^{n-\half})}{2}=g(\uu^n)+{\sf r}^n.
\end{equation}
Setting $w(t,x)=g(u(t,x))$ and using, again, the Taylor formula we have
%
%
%
\begin{equation}\label{USNORTH_21}
{\sf r}^n_j=\tfrac{1}{2}\,\tau^2\,\int_{0}^{\half}
\left[\,(\tfrac{1}{2}-s)\,w_{tt}(t_n+s\,\tau,x_j)
+s\,w_{tt}(t^{n-\half}+s\,\tau,x_j)\right]\;ds
\end{equation}
for $j=0,\dots,J+1$ and $n=1,\dots,N-1$, which, easily, yields
\begin{gather}
\max_{1\leq{n}\leq{\ssy N-1}}\|{\sf r}^n\|_{0,h}
+\max_{1\leq{n}\leq{\ssy N-1}}|{\sf r}^n|_{1,h}
\leq{\widehat{\sf C}}_{3,1}\,\tau^2,\label{SHELL_21}\\
\max_{2\leq{n}\leq{\ssy N-1}}\|{\sf r}^n-{\sf r}^{n-1}\|_{0,h}
+\max_{2\leq{n}\leq{\ssy N-1}}|{\sf r}^n-{\sf r}^{n-1}|_{1,h}
\leq{\widehat{\sf C}}_{3,2}\,\tau^3.\label{SHELL_22}
\end{gather}
%
%
%
%
%
\subsection{A Discrete Ellliptic Projection}
Let $v\in{\sf C}^2({\mathcal I};\rset)$. Then, we define ${\sf R}_h(v)\in\dspace$ (cf. \cite{AD1991})
by requiring
\begin{equation}\label{ELP_1}
\Delta_h({\sf R}_hv)=\PP(v'').
\end{equation}
Using the Taylor formula, it follows that
\begin{equation}\label{ELP_2}
\Delta_h(\PP v)-\PP(v'')=\tfrac{h^2}{12}\,{\sf r}^{\ssy{\rm E}}(v)
\end{equation}
where ${\sf r}^{\ssy{\rm E}}(v)\in\dspace$ is defined by
\begin{equation}\label{ELP_3}
({\sf r}^{\ssy{\rm E}}(v))_j:=\int_0^1\left[\,(1-y)^3\,v''''(x_j+h\,y)
+y^3\,v''''(x_{j-1}+h\,y)\,\right]\;dy,
\quad j=1,\dots,J.
\end{equation}
\par
First, subtract \eqref{ELP_1} from \eqref{ELP_2} to get
\begin{equation}\label{ELP_4}
\Delta_h(\PP v-{\sf R}_h v)=\tfrac{h^2}{12}\,{\sf r}^{\ssy{\rm E}}(v).
\end{equation}
Then, take the $(\cdot,\cdot)_{0,h}-$inner product of both sides of \eqref{ELP_4} with $(\PP v-{\sf R}_h v)$
and use \eqref{NewEra2}, the Cauchy-Schwarz inequality and \eqref{dpoincare} to obtain
\begin{equation}\label{ELP_5}
|{\sf R}_h v-\PP v|_{1,h}\leq\,\tfrac{{\sf L}}{12}\,h^2\,\|{\sf r}^{\ssy{\rm E}}(v)\|_{0,h}.
\end{equation}
\par
Finally, we use \eqref{ELP_5} to have
\begin{equation}\label{ELP_6}
\begin{split}
\Big|{\sf R}_h\left[\tfrac{\uu(t_{n+1},\cdot)-u(t_n,\cdot)}{\tau}\right]
-\left(\tfrac{u^{n+1}-u^n}{\tau}\right)\Big|_{1,h}\leq&\,\tfrac{\sf L}{12}\,h^2\,
\Big\|{\sf r}^{\ssy{\rm E}}\left[\tfrac{\uu(t_{n+1},\cdot)-u(t_n,\cdot)}{\tau}\right]\Big\|_{0,h}\\
\leq&\,\tfrac{{\sf L}^{\frac{3}{2}}}{12}
\,h^2\,\max_{\ssy{[0,T]\times{\mathcal I}}}|u_{txxxx}|,
\quad n=0,\dots,N-1.\\
\end{split}
\end{equation}
%
%
\section{Convergence Analysis}\label{Section4}
%
%
\subsection{A mollifier}
For $\delta>0$, let $\gff_{\delta}\in{\sf C}^3(\rset;\rset)$ (cf. \cite{KarMak},
\cite{Georgios1}) be an odd fuction defined by
\begin{equation}\label{ni_defin}
\gff_{\delta}(x):=\left\{
\begin{aligned}
&x,\hskip2.40truecm
\mbox{if}\ \ x\in[0,\delta],\\
&p_{\delta}(x),\hskip1.75truecm
\mbox{if}\ \ x\in (\delta,2\delta],\\
&2\,\delta,\hskip2.22truecm\mbox{if}\ \ x> 2\delta,\\
\end{aligned}
\right.\quad\forall\,x\ge 0,
\end{equation}
%
%
%
where $p_{\delta}$ is the unique polynomial of ${\mathbb P}^7[\delta,2\,\delta]$
that satisfies the following conditions:
\begin{equation*}
p_{\delta}(\delta)=\delta,\,\,\,p_{\delta}'(\delta)=1,
\,\,\,p_{\delta}''(\delta)=p_{\delta}'''(\delta)=0,
\,\,\,p_{\delta}(2\,\delta)=2\,\delta,
\,\,\,p_{\delta}'(2\,\delta)=p_{\delta}''(2\,\delta)=p_{\delta}'''(2\,\delta)=0.
\end{equation*}
%
\subsection{The (MBRFD) scheme}\label{The_MBRFD}
%
The modified version of the (BRFD) method (cf. \cite{Akrivis1}, \cite{KarMak},
\cite{Georgios1}) is a recursive procedure
that, for given $\delta>0$, derives approximations
$(V^n_{\delta})_{n=0}^{\ssy N}\subset\dspace$ of the solution
$\uu$ performing the steps below.
\par\vskip0.2truecm\par\noindent
{\tt Step 1}: Let $V_{\delta}^0\in\dspace$  be defined by
\begin{equation}\label{BR_CB1}
V_{\delta}^0:=\uu^0
\end{equation}
and $V_{\delta}^{\half}\in\dspace$ be specified by
\begin{equation}\label{BR_CB2}
\tfrac{V_{\delta}^{\half}-V_{\delta}^0}{(\tau/2)}
=\Delta_{h}\left(\,\tfrac{V_{\delta}^{\half}+V_{\delta}^0}{2}\,\right)
+g\big(\uu^0\big)\otimes\left(\tfrac{V_{\delta}^{\half}+V_{\delta}^0}{2}\right)
+\PP\left[\tfrac{f(t^{\half},\cdot)+f(t_0,\cdot)}{2}\right].
\end{equation}
\par\vskip0.2truecm\par\noindent
{\tt Step 2}: Define $\Phi_{\delta}^{\half}\in\gspace$ by
\begin{equation}\label{BR_CB3}
\Phi^{\half}_{\delta}:=g\big(\gff_{\delta}\big(V_{\delta}^{\half}\big)\big)
\end{equation}
and find $V_{\delta}^1\in\dspace$ such that
\begin{equation}\label{BR_CB4}
\tfrac{V^1_{\delta}-V^0_{\delta}}{\tau}=
\Delta_h\left(\tfrac{V^1_{\delta}+V^0_{\delta}}{2}\right)
+\gff_{\delta}\big(\Phi_{\delta}^{\half}\big)
\otimes\left(\tfrac{V^1_{\delta}+V^0_{\delta}}{2}\right)
+\PP\big[\tfrac{f(t_{1},\cdot)+f(t_0,\cdot)}{2}\big].
\end{equation}
\par\vskip0.2truecm\par\noindent
{\tt Step 3}: For $n=1,\dots,N-1$, first define
$\Phi_{\delta}^{n+\half}\in\gspace$ by
\begin{equation}\label{BR_CB5}
\Phi_{\delta}^{n+\half}:=2\,g\big(\gff_{\delta}\big(V_{\delta}^n\big)\big)
-\Phi_{\delta}^{n-\half}
\end{equation}
and, then, find $V^{n+1}_{\delta}\in\dspace$ such
that
\begin{equation}\label{BR_CB6}
\tfrac{V^{n+1}_{\delta}-V^n_{\delta}}{\tau}=
\Delta_h\left(\tfrac{V^{n+1}_{\delta}+V^n_{\delta}}{2}\right)
+\gff_{\delta}\big(\Phi_{\delta}^{n+\half}\big)
\otimes\left(\tfrac{V^{n+1}_{\delta}+V^n_{\delta}}{2}\right)
+\PP\big[\tfrac{f(t_{n+1},\cdot)+f(t_n,\cdot)}{2}\big].
\end{equation}
%
%
\subsection{Existence and uniqueness of the (MBRFD) approximations}
%
%
\begin{proposition}\label{BR_Exist}
Let $g_{\star}^0=\max_{\ssy{x\in\mathcal I}}|g(u_0(x))|$, $\delta\ge\,g_{\star}^0$
and ${\sf C}_{\delta}^{\ssy\sf BR, I}:=\frac{1}{4}\,\sup_{\ssy\rset}|\gff_{\delta}|$.
When $\tau\,{\sf C}_{\delta}^{\ssy\sf BR, I}\leq\tfrac{1}{2}$, then
the modified \text{\rm(BRFD)} approximations
%
%
are well-defined.
\end{proposition}
%
%
%
%
\begin{proof}
Let $\zeta\in\gspace$, $\varepsilon\in(0,1]$ and
${\sf T}_{\ssy{\sf BR}}:\dspace\rightarrow\dspace$
be a linear operator given by
\begin{equation*}
{\sf T}_{\ssy{\sf BR}}v:=2\,v-\,\varepsilon\,\tau\,\Delta_hv
-\varepsilon\,\tau\,\left[\,\gff_{\delta}(\zeta)\otimes v\,\right]
\quad\forall\,v\in\dspace.
\end{equation*}
Since $\delta\ge g_{\star}^0$, the definition of $\gff_{\delta}$ yields
that $\gff_{\delta}(g(u^0))=g(u^0)$. Thus, from \eqref{BR_CB2},
\eqref{BR_CB4} and \eqref{BR_CB6} it is easily seen that the
well-posedness of $V^{\half}_{\delta}$ and
$(V_{\delta}^n)_{n=1}^{\ssy N}$ follows easily by
securing the invertibility of ${\sf T}_{\ssy{\sf BR}}$.
Moving towards to this target, first we use \eqref{NewEra2} to obtain
\begin{equation}\label{dilesi}
\begin{split}
({\sf T}_{\ssy{\sf BR}}v,v)_{0,h}=&\,2\,\|v\|_{0,h}^2+\tau\,\varepsilon\,|v|_{1,h}^2
-\tau\,\varepsilon\,\left(\gff_{\delta}(\zeta)\otimes v,v\right)_{0,h}\\
\ge&\,2\,\|v\|_{0,h}^2+\tau\,\varepsilon\,|v|_{0,h}^2
-\tau\,\varepsilon\,\|v\|_{0,h}^2\,\,|\gff_{\delta}(\zeta)|_{\infty,h}\\
\ge&\,\tau\,\varepsilon\,|v|_{0,h}^2+4\,\|v\|_{0,h}^2\,\left(\,\tfrac{1}{2}-\tfrac{\tau}{4}
\,\max_{\ssy\rset}|\gff_{\delta}|\,\right)\\
\ge&\,\tau\,\varepsilon\,|v|_{1,h}^2+4\,\|v\|_{0,h}^2
\,\left(\,\tfrac{1}{2}-\tau\,{\sf C}_{\delta}^{\ssy\sf BR, I}\,\right)
\quad\forall\,v\in\dspace.\\
\end{split}
\end{equation}
Let us assume that $\tau\,{\sf C}_{\delta}^{\ssy\sf BR, I}\leq\tfrac{1}{2}$.
When $v\in{\rm Ker}({\sf T}_{\ssy{\sf BR}})$,
then $({\sf T}_{\ssy{\sf BR}}v,v)_{0,h}=0$, which, along with \eqref{dilesi},
yields $|v|_{1,h}=0$, or, equivalently, $v=0$. The latter argument
shows that ${\rm Ker}({\sf T}_{\ssy{\sf BR}})=\{0\}$ and,
thus, ${\sf T}_{\ssy{\sf BR}}$
is invertible, since $\dspace$ has finite dimension.
\end{proof}
\begin{remark}\label{remark_A}
Let us assume that $\tau\,{\sf C}_{\delta}^{\ssy{\sf BR, I}}\leq\tfrac{1}{2}$
and $\delta\ge g_{\star}^0$. Since $V_{\delta}^0:=u^0$ and $V^{\half}_{\delta}$
is well-defined, in view of \eqref{BR_CB2} and \eqref{BRS_12}, we conclude that
$U^{\half}$ is, also, well-defined and $U^{\half}=V^{\half}_{\delta}$.
\end{remark}
%
%
%
\subsection{Convergence of the (MBRFD) scheme}\label{Section_Conv_MBRFD}
%
In the theorem below, we investigate the convergence properties of the
modified (BRFD) approximations.
%
%
%
\begin{thm}\label{BR_CB_Conv}
Let $u_{\star}:=\max\limits_{\ssy{[0,T]\times{\mathcal I}}}|\uu|$,
$g_{\star}:=\max\limits_{\ssy{[0,T]\times{\mathcal I}}}|g(\uu)|$,
$\ddelta\ge\max\{u_{\star},g_{\star}\}$ and
$\tau\,{\sf C}^{\ssy{\sf BR,I}}_{\ddelta}\leq\tfrac{1}{2}$,
where ${\sf C}^{\ssy{\sf BR,I}}_{\ddelta}$ is the constant specified
in Proposition~\ref{BR_Exist}.
Then, there exist constants ${\sf C}^{\ssy{\sf BCV},1}_{\ssy\ddelta}\ge{\sf C}^{\ssy{\sf BR,I}}_{\ddelta}$,
${\sf C}^{\ssy{\sf BCV},2}_{\ddelta}>0$, ${\sf C}_{\ddelta}^{\ssy{\sf BCV},3}>0$
and ${\sf C}_{\ddelta}^{\ssy{\sf BCV},4}>0$,
independent of $\tau$ and $h$, such that: if
$\tau\,{\sf C}^{\ssy{\sf BCV},1}_{\ddelta}\leq\tfrac{1}{2}$, then
\begin{equation}\label{BR_CB_cnv_3}
|\uu^{\half}-V_{\ddelta}^{\half}|_{1,h}
\leq\,{\sf C}_{\ddelta}^{\ssy{\sf BCV},2}\,(\tau^2+\tau^{\half}\,h^2),
\end{equation}
\begin{equation}\label{BR_CB_cnv_1}
\max_{0\leq{m}\leq{\ssy N-1}}\|g(\uu^{m+\half})-\Phi_{\ddelta}^{m+\half}\|_{0,h}
+\max_{0\leq{m}\leq{\ssy N}}|\uu^m-V_{\ddelta}^m|_{1,h}
\leq\,{\sf C}_{\ddelta}^{\ssy{\sf BCV},3}\,(\tau^2+h^2)
\end{equation}
and
\begin{equation}\label{BR_CB_cnv_2}
\max_{0\leq{m}\leq{\ssy N-1}}\big|g(\uu^{m+\half})-\Phi_{\ddelta}^{m+\half}\big|_{1,h}
\leq\,{\sf C}_{\ddelta}^{\ssy{\sf BCV},4}\,(\tau^2+h^2).
\end{equation}
\end{thm}
%
%
%
\begin{proof}
To simplify the notation, we set
$\uu_{\star}^0:=\max_{\ssy{\mathcal I}}|u^0|$,
$\ee^{\half}:=\uu^{\half}-V_{\ddelta}^{\half}$,
$\ee^m:=\uu^m-V_{\ddelta}^m$
for $m=0,\dots,N$, and $\btheta^m:=g(\uu^{m+\half})-\Phi_{\ddelta}^{m+\half}$
for $m=0,\dots,N-1$.
In the sequel,  we will use the symbol $C$ to denote a generic constant that is
independent of $\tau$, $h$ and $\ddelta$, and may changes value from one line to the other.
Also, we will use the symbol $C_{\ddelta}$ to denote a generic constant that depends on
$\ddelta$ but is independent of $\tau$, $h$, and may changes value from one line to the other.
%
%
\par\noindent\vskip0.4truecm\par\noindent
$\boxed{{\tt Part\,\,\,1}:}$ Since $\ee^0=0$, after subtracting
\eqref{BR_CB2} from \eqref{NORAD_11} we obtain
\begin{equation}\label{BR_EE_a}
\ee^{\half}=\tfrac{\tau}{4}\,\Delta_h\ee^{\half}
+\tfrac{\tau}{4}\,\left[\,g(u^0)\otimes\ee^{\half}\,\right]
+\tfrac{\tau}{2}\,{\sf s}^{\quarter}.
\end{equation}
%
%
Next, take the $(\cdot,\cdot)_{0,h}-$inner product of
\eqref{BR_EE_a} with $\ee^{\half}$, and then use \eqref{NewEra2},
the Cauchy-Schwarz inequality, \eqref{USNORTH_01}, \eqref{SHELL_01a},
\eqref{SHELL_01b}, \eqref{SHELL_11} and the arithmetic mean inequality to get
\begin{equation*}
\begin{split}
\|\ee^{\half}\|_{0,h}^2+\tfrac{\tau}{4}\,|\ee^{\half}|_{1,h}^2
=&\,\tfrac{\tau}{4}\,(g(u^0)\otimes\ee^{\half},\ee^{\half})_{0,h}
+\tfrac{\tau}{2}\,({\sf s}^{\quarter},\ee^{\half})_{0,h}\\
\leq&\,\tfrac{\tau}{4}\,|g(u^0)|_{\infty,h}\,\|\ee^{\half}\|_{0,h}^2
+\tfrac{\tau}{2}\,\left[\,\|{\sf s}^{\quarter}-{\sf r}^{\quarter}\|_{0,h}
+\|{\sf r}^{\quarter}\|_{0,h}\,\right]\,\|\ee^{\half}\|_{0,h}\\
\leq&\,\tfrac{\tau}{4}\,|g(u^0)|_{\infty,h}\,\|\ee^{\half}\|_{0,h}^2
+C\,(\tau^2+\tau\,h^2)\,\|\ee^{\half}\|_{0,h}\\
\leq&\,\tfrac{\tau}{4}\,\max_{|x|\in[0,u^0_{\star}]}|g(x)|\,\|\ee^{\half}\|_{0,h}^2
+C\,(\tau^2+\tau\,h^2)^2+\tfrac{1}{2}\,\|\ee^{\half}\|_{0,h}^2.\\
\end{split}
\end{equation*}
%
%
Let ${\sf C}^{\ssy{\sf BR, II}}_{\ddelta}:=
\max\{\tfrac{1}{2}\,\max_{|x|\in[0,\uu^0_{\star}]}|g(x)|,
{\sf C}^{\ssy{\sf BR, I}}_{\ddelta}\}$ and
$\tau\,{\sf C}^{\ssy{\sf BR, II}}_{\ddelta}\leq\tfrac{1}{2}$.
Then, the inequality above yields that
\begin{equation}\label{BR_reco_1}
\|\ee^{\half}\|_{0,h}^2+\tau\,|\ee^{\half}|_{1,h}^2
\leq\,C\,(\tau^2+\tau\,h^2)^2.
\end{equation}
Taking the $(\cdot,\cdot)_{0,h}-$inner product of
\eqref{BR_EE_a} with $\Delta_h\ee^{\half}$,
and then using \eqref{NewEra2}, we obtain
\begin{equation}\label{BR_reco_2}
4\,|\ee^{\half}|_{1,h}^2+\tau\,\|\Delta_h\ee^{\half}\|_{0,h}^2
={\sf a}^1+{\sf a}^2,
\end{equation}
where
\begin{equation*}
\begin{split}
{\sf a}^1:=&\,-\,\tau\,(g(u^0)\otimes\ee^{\half},\Delta_h\ee^{\half})_{0,h},\\
{\sf a}^2:=&\,-2\,\tau\,(\eta^{\quarter},\Delta_h\ee^{\half})_{0,h}.\\
\end{split}
\end{equation*}
Now, we use the Cauchy-Schwarz inequality, the arithmetic mean inequality
and \eqref{BR_reco_1}, to have
\begin{equation}\label{BR_reco_3}
\begin{split}
{\sf a}^1\leq&\,\tau\,\max_{|x|\in[0,u^0_{\star}]}|g(x)|\,\|\ee^{\half}\|_{0,h}
\,\|\Delta_h\ee^{\half}\|_{0,h}\\
\leq&\,C\,\tau\,\|\ee^{\half}\|_{0,h}^2
+\tfrac{\tau}{6}\,\|\Delta_h\ee^{\half}\|_{0,h}^2\\
\leq&\,C\,\tau\,(\tau^2+\tau\,h^2)^2
+\tfrac{\tau}{6}\,\|\Delta_h\ee^{\half}\|_{0,h}^2.\\
\end{split}
\end{equation}
Also, \eqref{USNORTH_01}, the Cauchy-Schwarz inequality,
\eqref{NewEra1}, \eqref{SHELL_01a}, \eqref{SHELL_02b}, \eqref{SHELL_11} and
the arithmetic mean inequality, yield
\begin{equation}\label{BR_reco_4}
\begin{split}
{\sf a}^2=&\,-2\,\tau\,\big({\sf s}^{\quarter}-{\sf r}^{\quarter},\Delta_h\ee^{\half}\big)_{0,h}
-2\,\tau\,\big({\sf r}_{\ssy A}^{\quarter}-{\sf r}_{\ssy B}^{\quarter}-{\sf r}_{\ssy D}^{\quarter},
\Delta_h\ee^{\half}\big)_{0,h}+2\,\tau\,\big({\sf r}_{\ssy C}^{\quarter},\Delta_h\ee^{\half}\big)_{0,h}\\
\leq&\,2\,\tau\,\big[\,\|{\sf s}^{\quarter}-{\sf r}^{\quarter}\|_{0,h}+\|{\sf r}^{\quarter}_{\ssy A}\|_{0,h}
+\|{\sf r}^{\quarter}_{\ssy B}\|_{0,h}
+\|{\sf r}^{\quarter}_{\ssy D}\|_{0,h}\,\big]\,\|\Delta_h\ee^{\half}\|_{0,h}
-2\,\tau\,(\!\!(\delta_h{\sf r}^{\quarter}_{\ssy C},\delta_h\ee^{\half})\!\!)_{0,h}\\
\leq&\,C\,\tau\,(\tau^2+h^2)\,\|\Delta_h\ee^{\half}\|_{0,h}
+2\,\tau\,|{\sf r}^{\quarter}_{\ssy C}|_{1,h}\,|\ee^{\half}|_{1,h}\\
\leq&\,C\,\tau\,(\tau^2+h^2)\,\|\Delta_h\ee^{\half}\|_{0,h}
+C\,\tau^2\,|\ee^{\half}|_{1,h}\\
\leq&\,C\,\left[\,\tau\,(\tau^2+h^2)^2+\tau^4\,\right]
+\tfrac{\tau}{6}\,\|\Delta_h\ee^{\half}\|_{0,h}^2+|\ee^{\half}|_{1,h}^2.\\
\end{split}
\end{equation}
In view of \eqref{BR_reco_2}, \eqref{BR_reco_3} and \eqref{BR_reco_4}, we arrive at
\begin{equation}\label{BR_reco_5}
|\ee^{\half}|_{1,h}^2+\tau\,\|\Delta_h\ee^{\half}\|_{0,h}^2
\leq\,C\,(\tau^2+\tau^{\half}\,h^2)^2,
\end{equation}
which, obviously, yields \eqref{BR_CB_cnv_3}.
\par
Since $\ddelta\ge u_{\star}$, using \eqref{ni_defin}, \eqref{BR_CB3}
and \eqref{BR_reco_1}, we have
\begin{equation}\label{BR_reco_6}
\begin{split}
\|\btheta^0\|_{0,h}^2=&\,\big\|g(\gff_{\ddelta}(u^{\half}))
-g\big(\gff_{\ddelta}\big(V^{\half}_{\ddelta}\big)\big)\big\|_{0,h}^2\\
\leq&\,\sup_{\ssy\rset}|(g\circ\gff_{\ssy\ddelta})'|^2\,\|\ee^{\half}\|_{0,h}^2\\
\leq&\,C_{\ddelta}\,(\tau^2+\tau\,h^2)^2.\\
\end{split}
\end{equation}
Also, using Lemma~\ref{Lemma_D2}, \eqref{dpoincare} and \eqref{BR_reco_5}, we get
\begin{equation}\label{BR_reco_7}
\begin{split}
|\btheta^0|_{1,h}^2=&\,\big|g(\gff_{\ddelta}(u^{\half}))
-g\big(\gff_{\ddelta}\big(V^{\half}_{\ddelta}\big)\big)\big|_{1,h}^2\\
\leq&\,2\,\sup_{\ssy\rset}|(g\circ\gff_{\ssy\ddelta})'|^2\,|\ee^{\half}|_{1,h}^2
+2\,\sup_{\ssy\rset}|(g\circ\gff_{\ssy\ddelta})''|^2\,|\!|\!|\delta_hg(u^{\half})|\!|\!|_{\infty,h}^2
\,\|\ee^{\half}\|_{0,h}^2\\
\leq&\,C_{\ddelta}\,|\ee^{\half}|_{1,h}^2\\
\leq&\,C_{\ddelta}\,(\tau^2+\tau^{\half}\,h^2)^2.\\
\end{split}
\end{equation}
%
%
%
\par\noindent\vskip0.4truecm\par\noindent
$\boxed{{\tt Part\,\,\,2}:}$
We subtract \eqref{BR_CB4} and \eqref{BR_CB6} from \eqref{NORAD_12},
to obtain the following error equations:
\begin{equation}\label{BR_CB_EQ}
2\,(\ee^{n+1}-\ee^n)=\tau\,\Delta_h\left(\ee^{n+1}+\ee^n\right)
+\sum_{\kappa=1}^3{\sf Q}^{\kappa,n},\quad n=0,\dots,N-1,
\end{equation}
where
\begin{equation*}
\begin{split}
{\sf Q}^{1,n}:=&\,2\,\tau\,{\sf s}^{n+\frac{1}{2}},\\
{\sf Q}^{2,n}:=&\,\tau\,\gff_{\ddelta}
\big(\Phi_{\ddelta}^{n+\half}\big)
\otimes\left(\ee^{n+1}+\ee^{n}\right),\\
{\sf Q}^{3,n}:=&\,\tau\,\left[\,g(\uu^{n+\half})
-\gff_{\ddelta}\big(\Phi_{\ddelta}^{n+\half}\big)
\,\right]\otimes\left(\uu^{n+1}+\uu^{n}\right).\\
\end{split}
\end{equation*}
We take the inner product $(\cdot,\cdot)_{0,h}$
of \eqref{BR_CB_EQ} with $(\ee^{n+1}-\ee^{n})$,
and then, use \eqref{NewEra1}, to have
\begin{equation}\label{BR_CB_Gat2}
2\,\|\ee^{n+1}-\ee^{n}\|_{0,h}^2
+\tau\,\left[\,|\ee^{n+1}|_{1,h}^2
-|\ee^n|_{1,h}^2\,\right]
=\sum_{\kappa=1}^3{\sf q}^{\kappa,n},
\quad n=0,\dots,N-1,
\end{equation}
where
\begin{equation*}
{\sf q}^{\kappa,n}:=({\sf Q}^{\kappa,n},\ee^{n+1}-\ee^{n})_{0,h}.
\end{equation*}
\par
Let $n\in\{0\dots,N-1\}$.  Using the Cauchy-Schwarz ine\-quality,
the arithmetic mean ine\-quality, \eqref{SHELL_01a}
and \eqref{SHELL_11}, we have
\begin{equation}\label{BR_CB_Gat3}
\begin{split}
{\sf q}^{1,n}\leq&\,2\,\tau\,\left[\,\|{\sf s}^{n+\half}-{\sf r}^{n+\half}\|_{0,h}
+\|{\sf r}^{n+\half}\|_{0,h}\,\right]\,
\|\ee^{n+1}-\ee^{n}\|_{0,h}\\
\leq&\,2\,\tau\,(\tau^2+h^2)\,
\|\ee^{n+1}-\ee^{n}\|_{0,h}\\
\leq&\,C\,\tau^2\,(\tau^2+h^2)^2\,
+\tfrac{1}{6}\,\|\ee^{n+1}-\ee^{n}\|_{0,h}^2.
\end{split}
\end{equation}
Next, we use the Cauchy-Schwarz inequality,  \eqref{dpoincare},
\eqref{ni_defin} and the arithmetic mean inequality, to get
\begin{equation}\label{BR_CB_Gat4}
\begin{split}
{\sf q}^{2,n}\leq&\,\tau
\,|\gff_{\ddelta}(\Phi_{\ddelta}^{n+\half})|_{\infty,h}
\,\|\ee^{n+1}+\ee^{n}\|_{0,h}
\,\|\ee^{n+1}-\ee^{n}\|_{0,h}\\
\leq&\,C_{\ssy\ddelta}\,\tau\,
\,|\ee^{n+1}+\ee^{n}|_{1,h}
\,\|\ee^{n+1}-\ee^{n}\|_{0,h}\\
\leq&\,C_{\ssy\ddelta}\,\tau^2
\,\left[\,|\ee^{n+1}|_{1,h}^2+|\ee^{n}|_{1,h}^2\,\right]
+\tfrac{1}{6}\,\|\ee^{n+1}-\ee^{n}\|_{0,h}^2.\\
\end{split}
\end{equation}
Finally, taking into account that $\ddelta\ge g_{\star}$, we apply
the Cauchy-Schwarz inequality, \eqref{ni_defin}
and the arithmetic mean inequality to obtain
\begin{equation}\label{BR_CB_Gat5}
\begin{split}
{\sf q}^{3,n}\leq&\,2\,\tau\,\uu_{\star}\,
\|\gff_{\ddelta}(g(\uu^{n+\half}))-\gff_{\ddelta}(\Phi_{\ddelta}^{n+\half})\|_{0,h}
\,\|\ee^{n+1}-\ee^n\|_{0,h}\\
\leq&\,C\,\tau\,\max_{\ssy\rset}|\gff_{\ddelta}'|
\,\big\|g(\uu^{n+\half})-\Phi_{\ddelta}^{n+\half}\big\|_{0,h}
\,\|\ee^{n+1}-\ee^{n}\|_{0,h}\\
\leq&\,C_{\ddelta}\,\tau\,\|\btheta^{n}\|_{0,h}
\,\|\ee^{n+1}-\ee^{n}\|_{0,h}\\
\leq&\,C_{\ddelta}\,\tau^2\,\|\btheta^{n}\|_{0,h}^2
+\tfrac{1}{6}\,\|\ee^{n+1}-\ee^{n}\|_{0,h}^2.\\
\end{split}
\end{equation}
\par
From \eqref{BR_CB_Gat2}, \eqref{BR_CB_Gat3}, \eqref{BR_CB_Gat4}
and \eqref{BR_CB_Gat5}, we conclude that there exists a
constant ${\sf C}_{\ddelta}^{\ssy{\sf BR, III}}>0$, such that
\begin{equation}\label{BR_CB_Gat5.1}
\begin{split}
\|\ee^{n+1}-\ee^{n}\|_{0,h}^2
+\tau\,|\ee^{n+1}|_{1,h}^2\leq&\,\tau\,|\ee^n|_{1,h}^2
+{\sf C}_{\ddelta}^{\ssy{\sf BR, III}}\,\tau^2
\,\left[\,|\ee^{n+1}|_{1,h}^2+|\ee^{n}|_{1,h}^2
+\|\btheta^{n}\|_{0,h}^2\,\right]\\
&\,+C\,\tau^2\,(\tau^2+h^2)^2,\quad n=0,\dots,N-1.\\
\end{split}
\end{equation}
Let us find an error equation governing the midpoint error $\|\btheta^n\|_{0,h}$.
Subtracting \eqref{BR_CB5} from \eqref{NORAD_22} and using \eqref{ni_defin}
and the assumption $\ddelta\ge u_{\star}$, we obtain
\begin{equation}\label{BR_CB_Gat7.1}
\btheta^{n}+\btheta^{n-1}
=2\,\left[g(\gff_{\ddelta}(u^n))-g(\gff_{\ddelta}(V_{\ddelta}^n))\right]
+2\,{\sf r}^{n},\quad n=1,\dots,N-1,
\end{equation}
which, easily, yields that
\begin{equation}\label{BR_CB_Gat7}
\btheta^{n}-\btheta^{n-2}
=2\,{\sf R}^n+2\,({\sf r}^n-{\sf r}^{n-1}),\quad n=2,\dots,N-1,
\end{equation}
where ${\sf R}^n\in\dspace$ is defined by
\begin{equation}\label{BR_caracol_1}
{\sf R}^n:=g(\gff_{\ddelta}(V_{\ddelta}^{n-1}))-g(\gff_{\ddelta}(V_{\ddelta}^n))
-g(\gff_{\ddelta}(u^{n-1}))+g(\gff_{\ddelta}(u^n)).
\end{equation}
Then, we use \eqref{OPAP_1A}, \eqref{ni_defin} and the mean value theorem, to get
\begin{equation}\label{BR_RR_1}
\begin{split}
\|{\sf R}^n\|_{0,h}\leq&\,\sup_{\ssy\rset}|(g\circ\gff_{\ssy\ddelta})'|
\,\|\ee^n-\ee^{n-1}\|_{0,h}\\
&\quad+\sup_{\ssy\rset}|(g\circ\gff_{\ssy\ddelta})''|\,|\uu^{n-1}-\uu^{n}|_{\infty,h}
\,\left[\|\ee^n-\ee^{n-1}\|_{0,h}+\|\ee^{n}\|_{0,h}\right]\\
\leq&\,C_{\ddelta}\,\left[\,\|\ee^n-\ee^{n-1}\|_{0,h}
+\tau\,\|\ee^{n}\|_{0,h}\,\right],
\quad n=2,\dots,N-1.
\end{split}
\end{equation}
Taking the $(\cdot,\cdot)_{0,h}$ inner product of both sides of \eqref{BR_CB_Gat7}
with $\tau\big(\btheta^{n}+\btheta^{n-2}\big)$, and then using the
Cauchy-Schwarz inequality, \eqref{BR_RR_1}, \eqref{SHELL_22}
and \eqref{dpoincare}, it follows that
\begin{equation*}
\begin{split}
\tau\,\|\btheta^{n}\|_{0,h}^2-\tau\,\|\btheta^{n-2}\|^2_{0,h}\leq&\,
\left[\,2\,\tau\,\|{\sf R}^n\|_{0,h}
+2\,\tau\,\|{\sf r}^n-{\sf r}^{n-1}\|_{0,h}\,\right]
\,\|\btheta^{n}+\btheta^{n-2}\|_{0,h}\\
\leq&\,C_{\ddelta}\,\left[\,\tau\,\|\ee^n-\ee^{n-1}\|_{0,h}
+\tau^2\,\|\ee^{n}\|_{0,h}\,\right]\,\|\btheta^{n}+\btheta^{n-2}\|_{0,h}\\
&\quad+C\,\tau^4\,\|\btheta^{n}+\btheta^{n-2}\|_{0,h}\\
\leq&\,C_{\ddelta}\,\tau\,\|\ee^n-\ee^{n-1}\|_{0,h}
\,\|\btheta^n+\btheta^{n-2}\|_{0,h}\\
&\quad+C_{\ddelta}\,\tau^2\,|\ee^{n}|_{1,h}
\,\|\btheta^{n}+\btheta^{n-2}\|_{0,h}\\
&\quad+C\,\tau^4\,\|\btheta^{n}+\btheta^{n-2}\|_{0,h},
\quad n=2,\dots,N-1,\\
\end{split}
\end{equation*}
which, along with the application of the arithmetic mean inequality, yields
\begin{equation}\label{BR_CB_Gat8}
\begin{split}
\tau\,\|\btheta^{n}\|_{0,h}^2+\tau\,\|\btheta^{n-1}\|_{0,h}^2
\leq&\,\tau\,\|\btheta^{n-1}\|_{0,h}^2+\tau\,\|\btheta^{n-2}\|^2_{0,h}
+\|\ee^n-\ee^{n-1}\|_{0,h}^2+C\,\tau^6\\
&+C_{\ddelta}\,\tau^2\,\left[\,|\ee^{n}|_{1,h}^2
+\|\btheta^{n}\|_{0,h}^2+\|\btheta^{n-2}\|_{0,h}^2\,\right],
\quad n=2,\dots,N-1.\\
\end{split}
\end{equation}
\par
Thus, from \eqref{BR_CB_Gat5.1} and \eqref{BR_CB_Gat8}, we
conclude that there exists a constant ${\sf C}_{\ddelta}^{\ssy{\sf BR, IV}}>0$
such that:
\begin{equation}\label{BR_CB_Gat9}
(1-{\sf C}_{\ddelta}^{\ssy{\sf BR, IV}}\,\tau)\,{\sf Z}^{n+1}
\leq\,(1+{\sf C}_{\ddelta}^{\ssy{\sf BR, IV}}\,\tau)\,{\sf Z}^n
+C\,\tau^2\,(\tau^2+h^2)^2,\quad n=2,\dots,N-1,
\end{equation}
where
\begin{equation}\label{BR_Zerror}
{\sf Z}^m:=\|\ee^{m}-\ee^{m-1}\|_{0,h}^2
+\tau\,\left[\,|\ee^m|_{1,h}^2
+\|\btheta^{m-1}\|_{0,h}^2
+\|\btheta^{m-2}\|_{0,h}^2\,\right],
\quad n=2,\dots,N.
\end{equation}
Assuming that $\tau\,{\sf C}^{\ssy{\sf BR, V}}_{\ddelta}\leq\frac{1}{2}$ with
${\sf C}^{\ssy{\sf BR, V}}_{\ddelta}:=\max\{{\sf C}_{\delta}^{\ssy{\sf BR, III}},
{\sf C}_{\ddelta}^{\ssy{\sf BR, IV}}\}$,
a standard discrete Gronwall argument based on \eqref{BR_CB_Gat9} yields
\begin{equation}\label{BR_CB_Gat10}
\begin{split}
\max_{2\leq{m}\leq{\ssy N}}{\sf Z}^m
\leq&\,C_{\ddelta}\,\left[\,{\sf Z}^2+\tau\,(\tau^2+h^2)^2\,\right]\\
\leq&\,C_{\ddelta}\,\left[
\|\ee^{2}-\ee^{1}\|_{0,h}^2
+\tau\,|\ee^2|_{1,h}^2
+\tau\,\|\btheta^{1}\|_{0,h}^2
+\tau\,\|\btheta^{0}\|_{0,h}^2
+\tau\,(\tau^2+h^2)^2\,\right].\\
\end{split}
\end{equation}
Since $\ee^0=0$, after setting $n=0$ in \eqref{BR_CB_Gat5.1}
and then using \eqref{BR_reco_6}, we obtain
\begin{equation}\label{BR_reco_what_1}
\begin{split}
\|\ee^{1}\|_{0,h}^2+\tau\,|\ee^{1}|_{1,h}^2
\leq&\,C_{\ddelta}\,\left[\,\tau^2\,(\tau^2+h^2)^2
+\tau^2\,\|\btheta^0\|^2_{0,h}\,\right]\\
\leq&\,C_{\ddelta}\,\tau^2\,(\tau^2+h^2)^2.\\
\end{split}
\end{equation}
Also, setting $n=1$ in \eqref{BR_CB_Gat7.1}
and then using \eqref{BR_CB1}, we get
\begin{equation}\label{BR_Papagalini}
\btheta^1=-\btheta^0+2\,{\sf r}^1,
\end{equation}
which, along with \eqref{BR_reco_6}
and \eqref{SHELL_21}, yields
\begin{equation}\label{BR_reco_6.6.1}
\|\btheta^1\|_{0,h}^2\leq\,C_{\ddelta}\,(\tau^2+\tau\,h^2)^2.
\end{equation}
Also, setting $n=1$ in \eqref{BR_CB_Gat5.1}, and then
using \eqref{BR_reco_what_1} and \eqref{BR_reco_6.6.1},
we have
\begin{equation}\label{BR_reco_8}
\begin{split}
\|\ee^2-\ee^1\|_{0,h}^2+\,\tau\,|\ee^{2}|_{1,h}^2
\leq&\,C\,\tau^2\,(\tau^2+h^2)^2
+C_{\ddelta}\,\left[\,\tau\,|\ee^{1}|_{1,h}^2
+\tau^2\,\|\btheta^1\|_{0,h}^2\,\right]\\
\leq&\,C_{\ddelta}\,\tau^2\,(\tau^2+h^2)^2.\\
\end{split}
\end{equation}
Thus, \eqref{BR_CB_Gat10}, \eqref{BR_reco_8},
\eqref{BR_reco_6.6.1} and \eqref{BR_reco_6}
yield
\begin{equation}\label{BR_main_res_2}
\max_{2\leq{m}\leq{\ssy N}}{\sf Z}^m\leq\,C_{\ddelta}\,\tau\,(\tau^2+h^2)^2.
\end{equation}
Since $\ee^0=0$, \eqref{BR_CB_cnv_1}  follows, easily,
from \eqref{BR_Zerror}, \eqref{BR_main_res_2} and \eqref{BR_reco_what_1}.
%
%
%
%
\par\noindent\vskip0.4truecm\par\noindent
$\boxed{{\tt Part\,\,\,3}:}$
Let us define $\brho^m:={\sf R}_h[u(t^m,\cdot)]-\uu^m\in\dspace$ and
$\bhta^m:=V_{\ddelta}^m-{\sf R}_h[u(t^m,\cdot)]\in\dspace$ for $m=0,\dots,N$.
Then, using \eqref{BR_CB4}, \eqref{BR_CB6}, \eqref{NORAD_02}
and \eqref{ELP_1} we get
\begin{equation}\label{BR_NES_1}
2\,(\bhta^{n+1}-\bhta^n)=\tau\,\Delta_h\left(\bhta^{n+1}+\bhta^n\right)
+\sum_{\kappa=1}^{4}{\sf B}^{\kappa,n},\quad n=0,\dots,N-1,
\end{equation}
where
\begin{equation*}
\begin{split}
{\sf B}^{1,n}:=&\,2\,\tau\,\left(\,\tfrac{\uu^{n+1}-\uu^n}{\tau}
-{\sf R}_h\left[\tfrac{\uu(t^{n+1},\cdot)-\uu(t^n,\cdot)}{\tau}\right]\,\right),\\
{\sf B}^{2,n}:=&\,-2\,\tau\,{\sf r}^{n+\half},\\
{\sf B}^{3,n}:=&\,-\tau\,\gff_{\ddelta}\big(\Phi_{\ddelta}^{n+\half}\big)
\otimes(\ee^{n+1}+\ee^{n}),\\
{\sf B}^{4,n}:=&\,\tau\,\left[\,\gff_{\ddelta}\big(\Phi_{\ddelta}^{n+\half}\big)
-\gff_{\ddelta}(g(\uu^{n+\half}))\,\right]
\otimes(\uu^{n+1}+\uu^{n}).\\
\end{split}
\end{equation*}
Take the $(\cdot,\cdot)_{0,h}-$inner product of \eqref{BR_NES_1} with
$\Delta_h(\bhta^{n+1}-\bhta^{n})$, and then, use
\eqref{NewEra2} and \eqref{NewEra1}, to have
\begin{equation}\label{BR_NES_2}
2\,|\bhta^{n+1}-\bhta^n|_{1,h}^2
+\tau\left[\,\|\Delta_h\bhta^{n+1}\|_{0,h}^2
-\|\Delta_h\bhta^{n}\|_{0,h}^2\,\right]
=\sum_{\kappa=1}^4{\sf b}^{\kappa,n},\quad n=0,\dots,N-1,
\end{equation}
where
\begin{equation*}
{\sf b}^{\kappa,n}:=(\!(\delta_h{\sf B}^{\kappa,n},\delta_h(\bhta^{n+1}-\bhta^{n}))\!)_{0,h}.
\end{equation*}
\par
Let $n\in\{0\dots,N-1\}$.  Using the Cauchy-Schwarz inequality,
the arithmetic mean inequality,
\eqref{SHELL_02a} and \eqref{ELP_6}, we have
\begin{equation}\label{BR_NES_4}
\begin{split}
{\sf b}^{1,n}\leq&\,|{\sf B}^{1,n}|_{1,h}
\,|\bhta^{n+1}-\bhta^{n}|_{1,h}\\
\leq&\, C\,\tau\,h^2\,|\bhta^{n+1}-\bhta^{n}|_{1,h}\\
\leq&\, C\,\tau^2\,h^4+\tfrac{1}{6}\,|\bhta^{n+1}-\bhta^{n}|_{1,h}^2\\
\end{split}
\end{equation}
and
\begin{equation}\label{BR_NES_3}
\begin{split}
{\sf b}^{2,n}\leq&\,2\,\tau\,|{\sf r}^{n+\half}|_{1,h}\,
|\bhta^{n+1}-\bhta^{n}|_{1,h}\\
\leq&\,C\,\tau^3\,|\bhta^{n+1}-\bhta^{n}|_{1,h}\\
\leq&\,C\,\tau^6\,
+\tfrac{1}{6}\,|\bhta^{n+1}-\bhta^{n}|_{1,h}^2.\\
\end{split}
\end{equation}
Using, again, the Cauchy-Schwarz inequality
and the arithmetic mean inequality, we get
\begin{equation}\label{BR_NES_5}
{\sf b}^{3,n}+{\sf b}^{4,n}
\leq\tfrac{3}{2}\,\tau^2\,\left(\,|{\sf c}^{3,n}|_{1,h}^2+|{\sf c}^{4,n}|_{1,h}^2\,\right)
+\tfrac{2}{6}\,|\bhta^{n+1}-\bhta^{n}|_{1,h}^2
\end{equation}
where
\begin{equation*}
\begin{split}
{\sf c}^{3,n}:=&\,\gff_{\ddelta}(\Phi_{\ddelta}^{n+\half})\otimes(\ee^{n+1}+\ee^{n}),\\
{\sf c}^{4,n}:=&\,(g(\uu^{n+\half})-\gff_{\ddelta}(\Phi_{\ddelta}^{n+\half}))\otimes(\uu^{n+1}+\uu^{n}).\\
\end{split}
\end{equation*}
Then, we use \eqref{ni_defin}, \eqref{LH1},
\eqref{BR_CB_cnv_1}, \eqref{dpoincare}, \eqref{Gprop0} and the assumption
$\ddelta>g_{\star}$ to get
\begin{equation}\label{BR_NES_7}
\begin{split}
|{\sf c}^{3,n}|_{1,h}\leq&\,
\sup_{\ssy\rset}|\gff'_{\ddelta}|\,\big|\Phi_{\ddelta}^{n+\half}\big|_{1,h}
\,|\ee^{n+1}+\ee^{n}|_{\infty,h}
+\,\sup_{\ssy\rset}|\gff_{\ddelta}|
\,|\ee^{n+1}+\ee^n|_{1,h}\\
\leq&\,C_{\ddelta}\,\left[\,1+\big|\Phi_{\ddelta}^{n+\half}-g(u^{n+\half})\big|_{1,h}
+|g(u^{n+\half})|_{1,h}\,\right]
\,(|\ee^{n+1}|_{1,h}+|\ee^{n}|_{1,h})\\
\leq&\,C_{\ddelta}\,\left[\,1+\big|\Phi_{\ddelta}^{n+\half}-g(u^{n+\half})\big|_{1,h}\,\right]
\,(\tau^2+h^2)\\
\leq&\,C_{\ddelta}\,\left[\,|\btheta^n|_{1,h}+(\tau^2+h^2)\,\right]\\
\end{split}
\end{equation}
and
\begin{equation}\label{BR_NES_8}
\begin{split}
|{\sf c}^{4,n}|_{1,h}
\leq&\,C\,\left[\,\|\gff_{\ddelta}(g(\uu^{n+\half}))
-\gff_{\ddelta}(\Phi_{\ddelta}^{n+\half})\|_{0,h}
+\big|\gff_{\ddelta}(g(\uu^{n+\half}))
-\gff_{\ddelta}(\Phi_{\ddelta}^{n+\half})\big|_{1,h}\,\right]\\
\leq&\,C\,\big|\gff_{\ddelta}(\Phi_{\ddelta}^{n+\half})
-\gff_{\ddelta}(g(\uu^{n+\half}))\big|_{1,h}\\
\leq&\,C\,\left[\,\sup_{\ssy\rset}|\gff'_{\ddelta}|\,|\btheta^n|_{1,h}
+\max_{\ssy\rset}|\gff''_{\ddelta}|\,|\!|\!|\delta_h(g(u^{n+\half}))|\!|\!|_{\infty,h}
\,\|\btheta^n\|_{0,h}\,\right]\\
\leq&\,C_{\ddelta}\,\left[\,|\btheta^n|_{1,h}+(\tau^2+h^2)\,\right].\\
\end{split}
\end{equation}
Thus, \eqref{BR_NES_5}, \eqref{BR_NES_7} and \eqref{BR_NES_8} yield
\begin{equation}\label{BR_NES_88}
{\sf b}^{3,n}+{\sf b}^{4,n}
\leq C_{\ddelta}\,\tau^2\,\left[\,|\btheta^n|_{1,h}^2+(\tau^2+h^2)^2\,\right]
+\tfrac{2}{6}\,|\bhta^{n+1}-\bhta^{n}|_{1,h}^2.
\end{equation}
\par
From \eqref{BR_NES_2}, \eqref{BR_NES_4}, \eqref{BR_NES_3}
and \eqref{BR_NES_88}, we conclude that there
exists a constant ${\sf C}_{\ddelta}^{\ssy{\sf BR, VI}}>0$, such that
\begin{equation}\label{BR_NES_9}
\begin{split}
|\bhta^{n+1}-\bhta^{n}|_{1,h}^2
+\tau\,\|\Delta_h\bhta^{n+1}\|_{0,h}^2\leq&\,\tau\,\|\Delta_h\bhta^n\|_{0,h}^2
+{\sf C}_{\ddelta}^{\ssy{\sf BR, VI}}\,\tau^2\,|\btheta^{n}|_{1,h}^2\\
&\quad+{\sf C}_{\ddelta}^{\ssy{\sf BR, VI}}\,
\tau^2\,(\tau^2+h^2)^2,\quad n=0,\dots,N-1.\\
\end{split}
\end{equation}
Taking the $(\cdot,\cdot)_{0,h}$ inner product of both sides of \eqref{BR_CB_Gat7} by
$\tau\,\Delta_h(\btheta^{n}+\btheta^{n-2})$, and using \eqref{NewEra1},
the Cauchy-Schwarz inequality and \eqref{SHELL_22}, we have
\begin{equation}\label{BR_NES_10}
\begin{split}
\tau\,|\btheta^{n}|_{1,h}^2-\tau\,|\btheta^{n-2}|^2_{1,h}
=&\,2\,\tau\,(\!\!(\delta_h{\sf R}^n,
\delta_h(\btheta^{n}+\btheta^{n-2}))\!\!)_{0,h}
+2\,\tau\,(\!\!(\delta_h({\sf r}^n-{\sf r}^{n-1}),
\delta_h(\btheta^{n}+\btheta^{n-2}))\!\!)_{0,h}\\
\leq&\,2\,\tau\,\left(\,|{\sf R}^n|_{1,h}+|{\sf r}^n-{\sf r}^{n-1}|_{1,h}\right)
\,(|\btheta^{n}|_{1,h}+|\btheta^{n-2}|_{1,h})\\
\leq&\,2\,\tau\,\left(\,|{\sf R}^n|_{1,h}+\tau^3\right)
\,(|\btheta^{n}|_{1,h}+|\btheta^{n-2}|_{1,h}),
\quad n=2,\dots,N-1.\\
\end{split}
\end{equation}
%
%
Using \eqref{BR_caracol_1}, \eqref{OPAP_1B}, \eqref{dpoincare},
\eqref{BR_CB_cnv_1} and \eqref{ELP_6}, we get
\begin{equation}\label{BR_NES_11b}
\begin{split}
|{\sf R}^n|_{1,h}\leq&\,C_{\ddelta}\,\left[\,|\ee^n-\ee^{n-1}|_{1,h}+\tau\,(\tau^2+h^2)\,\right]\\
\leq&\,C_{\ddelta}\,\left[\,|\bhta^n-\bhta^{n-1}|_{1,h}
+|\brho^n-\brho^{n-1}|_{1,h}+\tau\,(\tau^2+h^2)\,\right]\\
\leq&\,C_{\ddelta}\,\left[\,|\bhta^n-\bhta^{n-1}|_{1,h}+\tau\,(\tau^2+h^2)\,\right],
\quad n=2,\dots,N-1.
\end{split}
\end{equation}
Then, \eqref{BR_NES_10}, \eqref{BR_NES_11b} and the arithmetic mean inequality, yield
\begin{equation}\label{BR_NES_12b}
\begin{split}
\tau\,|\btheta^{n}|_{1,h}^2+\tau\,|\btheta^{n-1}|_{1,h}^2\leq&\,\tau\,|\btheta^{n-1}|_{1,h}^2
+\tau\,|\btheta^{n-2}|^2_{1,h}\\
&\quad+C_{\ddelta}\,\tau\,
\left[\,|\bhta^n-\bhta^{n-1}|_{1,h}+\tau\,(\tau^2+h^2)\right]
\,(|\btheta^{n}|_{1,h}+|\btheta^{n-2}|_{1,h})\\
\leq&\,\tau\,|\btheta^{n-1}|_{1,h}^2+\tau\,|\btheta^{n-2}|^2_{1,h}+|\bhta^n-\bhta^{n-1}|_{1,h}^2\\
&\quad+C_{\ddelta}\,\left[\,\tau^2\,
\,(|\btheta^{n}|_{1,h}^2+|\btheta^{n-2}|_{1,h}^2)+\tau^2\,(\tau^2+h^2)^2\,\right],
\quad n=2,\dots,N-1.\\
\end{split}
\end{equation}
\par
Combining \eqref{BR_NES_9} and \eqref{BR_NES_12b}, we conclude that there exists
a positive constant ${\sf C}_{\ddelta}^{\ssy{\sf BR, VII}}$ such that:
\begin{equation}\label{BR_NES_11}
(1-{\sf C}_{\ddelta}^{\ssy{\sf BR, VII}}\,\tau)\,{\sf Z}_{\star}^{n+1}
\leq\,(1+{\sf C}_{\ddelta}^{\ssy{\sf BR, VII}}\,\tau)\,{\sf Z}_{\star}^{n}
+C_{\ddelta}\,\tau^2\,(\tau^2+h^2)^2,\quad n=2,\dots,N-1,
\end{equation}
where
\begin{equation}\label{BR_NES_10_a}
{\sf Z}_{\star}^m:=|\bhta^m-\bhta^{m-1}|^2_{1,h}+\tau\,\|\Delta_h\bhta^m\|_{0,h}^2
+\tau\,|\btheta^{m-1}|_{1,h}^2+\tau\,|\btheta^{m-2}|_{1,h}^2,\quad m=2,\dots,N.
\end{equation}
Assuming that $\tau\,{\sf C}^{\ssy{\sf BR, VIII}}_{\ddelta}\leq\frac{1}{2}$, where
${\sf C}^{\ssy{\sf BR, VIII}}_{\ddelta}:=\max\{{\sf C}_{\ddelta}^{\ssy{\sf BR, VII}},
{\sf C}_{\ddelta}^{\ssy{\sf BR, VI}}\}$, and using a standard discrete Gronwall argument
based on \eqref{BR_NES_11}, we obtain
\begin{equation}\label{BR_NES_12}
\begin{split}
\max_{2\leq{m}\leq{\ssy N}}{\sf Z}_{\star}^m\leq&\,C_{\ddelta}
\,\left[\,{\sf Z}_{\star}^2+\tau\,(\tau^2+h^2)^2\,\right]\\
\leq&\,C_{\ddelta}
\,\left[\,|\bhta^2-\bhta^1|_{1,h}^2+\tau\,\|\Delta_h\bhta^2\|_{0,h}^2+\tau\,|\btheta^1|_{1,h}^2
+\tau\,|\btheta^0|_{1,h}^2+\tau\,(\tau^2+h^2)^2\,\right].\\
\end{split}
\end{equation}
After setting $n=0$ in \eqref{BR_NES_9} and then
using  \eqref{BR_reco_7} and \eqref{ELP_4}, we obtain
\begin{equation}\label{BR_NES_13}
\begin{split}
|\bhta^{1}-\bhta^0|_{1,h}^2+\tau\,\|\Delta_h\bhta^1\|_{0,h}^2
\leq&\,\tau\,\|\Delta_h\bhta^0\|_{0,h}^2+C_{\ddelta}\,\left[\,
\tau^2\,|\btheta^0|_{1,h}^2
+\tau^2\,(\tau^2+h^2)^2\,\right]\\
\leq&\,C\,\tau\,h^4+C_{\ddelta}\,\tau^2\,(\tau^2+h^2)^2\\
\leq&\,C_{\ddelta}\,\tau\,(\tau^2+h^2)^2.\\
\end{split}
\end{equation}
Using \eqref{BR_Papagalini}, \eqref{BR_reco_7} and \eqref{SHELL_21},
we have
%
\begin{equation}\label{BR_NES_14}
\begin{split}
|\btheta^1|_{1,h}^2
\leq&\,C_{\ddelta}\,(\tau^2+h^2)^2.\\
\end{split}
\end{equation}
Set  $n=1$ in \eqref{BR_NES_9} to conclude that
\begin{equation*}
|\bhta^2-\bhta^1|_{1,h}^2+\tau\,\|\Delta_h\bhta^{2}\|_{0,h}^2
\leq\tau\,\|\Delta_h\bhta^{1}\|_{0,h}^2
+C_{\ddelta}\,\left[\,\tau^2\,(\tau^2+h^2)^2
+\tau^2\,|\btheta^1|_{1,h}^2\,\right]
\end{equation*}
which, along with, \eqref{BR_NES_13} and \eqref{BR_NES_14}, yields
\begin{equation}\label{BR_NES_15}
|\bhta^2-\bhta^1|_{1,h}^2
+\tau\,\|\Delta_h\bhta^{2}\|_{0,h}^2\leq\,C_{\ddelta}
\,\tau\,(\tau^2+h^2)^2.
\end{equation}
Thus, from \eqref{BR_NES_12}, \eqref{BR_NES_15}, \eqref{BR_NES_14}
and \eqref{BR_reco_7}, we obtain
\begin{equation}\label{BR_NES_16}
\max_{2\leq{m}\leq{\ssy N}}{\sf Z}_{\star}^m\leq\,C_{\ddelta}\,\tau\,(\tau^2+h^2)^2.
\end{equation}
Finally, \eqref{BR_CB_cnv_2} follows, easily, from \eqref{BR_NES_10_a}
and \eqref{BR_NES_16}.
\end{proof}
%
%
%
\subsection{Convergence of the (BRFD) method}
%
%
%
%
\begin{thm}\label{DR_Final}
Let $u_{\star}:=\max\limits_{\ssy{[0,T]\times{\mathcal I}}}|\uu|$,
$g_{\star}:=\max\limits_{\ssy{[0,T]\times{\mathcal I}}}|g(\uu)|$,
$\ddelta\ge\,2\,\max\{u_{\star},g_{\star}\}$, ${\sf C}^{\ssy{\sf BR,I}}_{\ddelta}$
be the constant determined in Proposition~\ref{BR_Exist}, ${\sf C}^{\ssy{\sf BCV},1}_{\ddelta}$,
${\sf C}^{\ssy{\sf BCV},2}_{\ddelta}$, ${\sf C}_{\ddelta}^{\ssy{\sf BCV},3}$ and
${\sf C}_{\ddelta}^{\ssy{\sf BCV},4}$ be the constants
specified in Theorem~\ref{BR_CB_Conv}, where
${\sf C}^{\ssy{\sf BCV},1}_{\ddelta}\ge{\sf C}^{\ssy{\sf BR, I}}_{\ddelta}$.
If
\begin{equation}\label{BR_Xmesh}
\tau\,{\sf C}^{\ssy{\sf BCV},1}_{\ddelta}\leq\tfrac{1}{2},
\quad
{\sf C}_{\ddelta}^{\ssy{\sf BCV},2}\,\sqrt{{\sf L}}\,(\tau^2+\tau^{\half}\,h^{2})
\leq\,\tfrac{\ddelta}{2},
\quad
{\sf C}_{\ddelta}^{\ssy{\sf BCV},4}\,\sqrt{{\sf L}}\,(\tau^2+h^{2})
\leq\,\tfrac{\ddelta}{2},
\end{equation}
then, the method (BRFD) is well-defined and the following error estimates hold
\begin{equation}\label{BR_true_2}
|\uu^{\half}-\UUU^{\half}|_{1,h}
\leq\,{\sf C}_{\ddelta}^{\ssy{\sf BCV},2}\,(\tau^2+\tau^{\half}\,h^2)
\end{equation}
and
\begin{equation}\label{BR_true_1}
\max_{0\leq{m}\leq{\ssy N-1}}|g(\uu^{m+\half})-\Phi^{m+\half}|_{1,h}
+\max_{0\leq{m}\leq{\ssy N}}|\uu^m-\UUU^m|_{1,h}
\leq\,\max\{{\sf C}_{\ddelta}^{\ssy{\sf BCV},3},
{\sf C}_{\ddelta}^{\ssy{\sf BCV},4}\}\,(\tau^2+h^2).
\end{equation}
\end{thm}
%
%
%
%
\begin{proof}
Since $\ddelta\ge 2\,\max\{g_{\star},u_{\star}\}$, the convergence estimates
\eqref{BR_CB_cnv_3} and \eqref{BR_CB_cnv_2},
the discrete Sobolev inequality \eqref{LH1}
and the mesh size conditions \eqref{BR_Xmesh} imply that the (MBRFD) are
well-defined and
\begin{equation*}
\begin{split}
\big|\Phi_{\ddelta}^{n+\half}\big|_{\infty,h}\leq&\,\big|g(\uu^{n+\half})
-\Phi^{n+\half}_{\ddelta}\big|_{\infty,h}+\big|g(\uu^{n+\half})|_{\infty,h}\\
\leq&\,{\sqrt{\sf L}}\,\big|g(\uu^{n+\half})
-\Phi^{n+\half}_{\ddelta}\big|_{1,h}+g_{\star}\\
\leq&\,{\sf C}_{\ddelta}^{\ssy{\sf BCV},4}\,\sqrt{{\sf L}}\,(\tau^2+h^{2})
+\tfrac{\ddelta}{2}\\
\leq&\,\ddelta,
\quad n=0,\dots,N-1,
\end{split}
\end{equation*}
and
\begin{equation*}
\begin{split}
\big|V_{\ddelta}^{\half}\big|_{\infty,h}\leq&\,\big|\uu^{\half}
-V^{\half}_{\ddelta}\big|_{\infty,h}+\big|\uu^{\half}|_{\infty,h}\\
\leq&\,{\sqrt{\sf L}}\,\big|\uu^{n+\half}
-V^{\half}_{\ddelta}\big|_{1,h}+u_{\star}\\
\leq&\,{\sf C}_{\ddelta}^{\ssy{\sf BCV},2}\,\sqrt{{\sf L}}
\,(\tau^2+\tau^{\half}\,h^{2})+\tfrac{\ddelta}{2}\\
\leq&\,\ddelta,
\end{split}
\end{equation*}
which, along with \eqref{ni_defin}, yield
\begin{equation}\label{BR_crucial_1}
\gff_{\ddelta}(V^{\half}_{\ddelta})=V^{\half}_{\ddelta},\quad
\gff_{\ddelta}\big(\Phi_{\ddelta}^{n+\half}\big)
=\Phi_{\ddelta}^{n+\half},\quad n=0,\dots,N-1.
\end{equation}
Thus, the (MBRFD) approximations are
(BRFD) approximations when $\delta=\ddelta$, i.e.
\eqref{BRS_1}-\eqref{BRS_4} hold after replacing
$\UUU^{\half}$ by $V_{\ddelta}^{\half}$,
$\UUU^n$ by $V^n_{\ddelta}$ for $n=0,\dots,N$, and
$\Phi^{n+\half}$ by $\Phi^{n+\half}_{\ddelta}$ for $n=0,\dots,N-1$.
\par
Let $\UUU^{\half}$, $(\UUU^n)_{n=0}^{\ssy N}$ and $(\Phi^{n+\half})_{n=0}^{\ssy N-1}$
be approximations derived by the (BRFD) method. Then, we introduce the errors
${\mathfrak q}^{\half}:=V^{\half}_{\ddelta}-W^{\half}$,
${\mathfrak q}^n:=V_{\ddelta}^n-W^n$ for $n=0,\dots,N$, and
${\mathfrak q}^{n}_{\star}:=\Phi_{\ddelta}^{n+\half}-\Phi^{n+\half}$ for
$n=0,\dots,N-1$.
Since $\tau\,{\sf C}_{\ddelta}^{\ssy{\sf BR, I}}\leq\tfrac{1}{2}$
and $\ddelta\ge g_{\star}\ge g_{\star}^0$, Remark~\ref{remark_A} and \eqref{BRS_1} yield
${\mathfrak q}^0=0$, ${\mathfrak q}^{\half}=0$ and ${\mathfrak q}^{0}_{\star}=0$.
%
%
Now, we assume that for a given $m\in\{0,\dots,N-1\}$ it holds that ${\mathfrak q}^m=0$
and ${\mathfrak q}_{\star}^{m}=0$. Subracting \eqref{BRS_4} from
\eqref{BR_CB6} (or \eqref{BRS_2} from \eqref{BR_CB4} when $m=0$),
and then using \eqref{BR_crucial_1}, we obtain
\begin{equation}\label{BR_exu_DR_1}
{\mathfrak q}^{m+1}=\tfrac{\tau}{2}\,\Delta_h{\mathfrak q}^{m+1}
+\tfrac{\tau}{2}\,\left[\gff_{\ddelta}\big(\Phi_{\ddelta}^{m+\half}\big)
\otimes{\mathfrak q}^{m+1}\right].
\end{equation}
Next, taking the inner product $(\cdot,\cdot)_{0,h}$ with ${\mathfrak q}^{m+1}$ and
then using \eqref{NewEra2}, the Cauchy-Schwarz inequality, \eqref{ni_defin}
and the definion of ${\sf C}_{\ddelta}^{\ssy{\sf BR, I}}$, we get
\begin{equation*}
\begin{split}
0=&\,\|{\mathfrak q}^{m+1}\|_{0,h}^2+\tfrac{\tau}{2}\,|{\mathfrak q}^{m+1}|_{1,h}^2
-\tfrac{\tau}{2}\,\big(\gff_{\ddelta}(\Phi_{\ddelta}^{m+\half})
\otimes{\mathfrak q}^{m+1},{\mathfrak q}^{m+1}\big)_{0,h}\\
\ge&\,\tfrac{\tau}{2}\,|{\mathfrak q}^{m+1}|_{1,h}^2
+2\,\|{\mathfrak q}^{m+1}\|_{0,h}^2
\,\left(\tfrac{1}{2}-\tfrac{\tau}{4}\,\sup_{\ssy\rset}|\gff_{\ddelta}|\,\right)\\
\ge&\,\tfrac{\tau}{2}\,|{\mathfrak q}^{m+1}|_{1,h}^2
+2\,\|{\mathfrak q}^{m+1}\|_{0,h}^2
\,\left(\tfrac{1}{2}-\tau\,{\sf C}_{\ddelta}^{\ssy{\sf BR, I}}\,\right)\\
\ge&\,\tfrac{\tau}{2}\,|{\mathfrak q}^{m+1}|_{1,h}^2,
\end{split}
\end{equation*}
which, obviously, yields that ${\mathfrak q}^{m+1}=0$. When $m\leq N-2$, observing that
\begin{equation*}
{\mathfrak q}_{\star}^{m+1}=2\,\left[g(V^{m+1}_{\ddelta})-g(\UUU^{m+1})\right]
-{\mathfrak q}_{\star}^{m},
\end{equation*}
we arrive at ${\mathfrak q}_{\star}^{m+1}=0$. The induction argument above, shows that,
under our assumptions the (BRFD) approximations are those derived from of the (MBRFD) scheme
when $\delta=\ddelta$, and thus the error estimates \eqref{BR_true_2} and \eqref{BR_true_1}
follow as a natural outcome of \eqref{BR_CB_cnv_3}, \eqref{BR_CB_cnv_1}
and \eqref{BR_CB_cnv_2}.
\end{proof}
\begin{remark}
Let us make the choice $\Phi^{\half}=g(u^0)$  (see \cite{Besse1}, \cite{Besse2})
instead of \eqref{BRS_13}. Then, we obtain
$\|\btheta^0\|_{0,h}={\mathcal O}(\tau)$, $\|\btheta^1\|_{0,h}={\mathcal O}(\tau)$
and ${\sf Z}^2={\mathcal O}(\tau\,(\tau+h^2)^2)$. Thus, from \eqref{BR_CB_Gat10}
we arrive at a suboptimal error estimate of the form ${\mathcal O}(\tau+h^2)$.
Here, we skip the problem
by introducing \eqref{BRS_12} (cf. \cite{Georgios1}) that derives a higher order approximation $\Phi^{\half}$ of
$g(u(t_1,\cdot))$.
\end{remark}
%
%
%
%
%
%
%
%

\end{document}